\documentclass[aoap]{imsart}
\usepackage{caption}

\usepackage{amsmath,amssymb,mathrsfs,amsthm,amsfonts,mathtools}
\usepackage[inline]{enumitem} 
\usepackage[usenames,dvipsnames]{xcolor}
\usepackage{hyperref}
\usepackage{tikz}
\usepackage{bbm}
\tikzstyle{block} = [rectangle, draw, 
   ]
\usetikzlibrary{shapes,positioning}

\usepackage{comment}
\hypersetup{%
	colorlinks=true, linkcolor=ForestGreen,
	citecolor=ForestGreen
}
\usepackage{acronym}
\usepackage{dsfont}

\mathtoolsset{showonlyrefs}

\newcommand{\vvert}[1]{{\left\vert\kern-0.25ex\left\vert\kern-0.25ex\left\vert #1 
    \right\vert\kern-0.25ex\right\vert\kern-0.25ex\right\vert}}
\usepackage{amsmath}

\DeclareMathOperator*{\argmin}{arg\,min}

\newcommand{\R}{\mathds{R}}

\renewcommand{\S}{\mathcal{S}}

\newcommand\inner[2]{\langle #1,#2\rangle}

\newtheorem{stat}{Statement}[section]

\newtheorem{theorem}[stat]{Theorem}
\newtheorem{lemma}[stat]{Lemma}
\theoremstyle{definition}
\newtheorem{definition}[stat]{Definition}
\newtheorem{remark}[stat]{Remark}

\numberwithin{equation}{section}


\begin{document}
\begin{frontmatter}
\title{Entropic Selection Principle for Monge's Optimal Transport}
\runtitle{Entropic Selection Principle}
\begin{aug}
\author[A]{\fnms{Shrey} \snm{Aryan}\ead[label=e1]{shrey183@mit.edu}}\and
\author[B]{\fnms{Promit} \snm{Ghosal}\ead[label=e2]{promit@uchicago.edu}}
\address[A]{Department of Mathematics, Massachusetts Institute of Technology, \printead{e1}}
\address[B]{Department of Statistics, University of Chicago, \printead{e2}}
\end{aug}

\begin{abstract}
 We investigate the small regularization limit of entropic optimal transport with the Euclidean cost in dimensions $d>1$, where the marginal measures are supported on disjoint compact sets, and are absolutely continuous with respect to the Lebesgue measure with continuous densities. Our results establish that the limiting optimal transport plan is supported on \emph{transport rays}. Furthermore, within each transport ray, the limiting transport plan uniquely minimizes a relative entropy functional with respect to specific reference measures supported on the rays. This provides a complete and unique characterization of the limiting transport plan. While similar results have been obtained for \(d = 1\) in \cite{Marino} and for discrete measures in \cite{peyré2020computationaloptimaltransport}, this work resolves the previously open case in higher dimensions $d>1.$  
\end{abstract}

\end{frontmatter}
\section{Introduction}\label{sec_intro}

Optimal transport has emerged as one of the most prominent mathematical fields in recent decades, forging deep connections with various branches of mathematics. The Monge-Kantorovich optimal transport problem aims to transform one probability measure into another while minimizing a specified transportation cost. In Monge's original formulation, the objective was to minimize the average distance transported. This leads to the challenging problem of finding an optimal transport map that pushes one distribution onto another while preserving mass. The existence and properties of Monge’s optimal transport map have been established under various assumptions through seminal works of Caffarelli, Feldman, and McCann \cite{mcann-caff-feldman} and Evans and Gangbo \cite{evans-gangbo}. In particular, Caffarelli, Feldman, and McCann constructed Monge’s transport map along one-dimensional structures known as transport rays by analyzing the limits of solutions to relaxed optimal transport problems that approximate the original Monge formulation. This approach was subsequently extended to the setting of Riemannian manifolds in~\cite{manifold}.

Building upon the classical theory of Monge's problem, a significant development was introduced by Klartag~\cite{klartag2017needle}, who established a striking connection between optimal transport and localization techniques in convex and Riemannian geometry. The central insight of this approach is that certain high-dimensional geometric and functional inequalities can be effectively studied by reducing them to families of one-dimensional problems along so-called needles—i.e., transport rays endowed with appropriate measures. This localization paradigm has a rich legacy in geometric analysis, underpinning landmark results such as those of Payne and Weinberger~\cite{payne1960optimal}, Gromov~\cite{gromov1987generalization}, Lovász and Simonovits~\cite{lovasz1993random}, and Kannan, Lovász, and Simonovits~\cite{kannan1995isoperimetric}. Klartag’s work extends this powerful idea to Riemannian manifolds satisfying the curvature–dimension condition $\mathsf{CD}(\kappa,d)$ by deeply analyzing the structure of Kantorovich potentials and transport rays in such spaces.

While the classical optimal transport theory has achieved considerable success, its direct implementation is often computationally prohibitive in high-dimensional settings due to the non-smooth nature of the transport problem. A significant driving force behind the recent interest in optimal transport has been the development and popularization of entropic optimal transport, which offers a computationally efficient framework for approximating classical optimal transport problems. The entropic optimal transport problem, pioneered by Cuturi \cite{Cuturi}, modifies the Kantorovich formulation by penalizing the coupling’s deviation from a reference measure through a relative entropy term, enabling faster and more stable numerical algorithms. When the regularization parameter \(\varepsilon > 0\) is small, the entropic optimal transport provides an approximation of the classical optimal transport map.

Given its computational advantages, a natural theoretical question arises: which optimal transport map is selected as the entropic regularization parameter \(\varepsilon\) approaches zero, especially in Monge's optimal transport problem? Understanding this limiting behavior is crucial for bridging the gap between the entropic and classical optimal transport theories. While the existence of limits has been established under various settings, the precise characterization of the limiting transport map remains an active area of research, especially when the transport cost lacks strict convexity or smoothness. In this paper, we contribute to this line of work by providing a detailed variational characterization of the entropic limit, and by exploring entropic optimal transport through a localization lens—analyzing the entropic optimal coupling via restrictions to shrinking neighborhoods along transport rays, thereby establishing a localization principle tailored to the entropic setting.

Given two probability measures, $\mu, \nu \in \mathcal{P}(\R^d)$, the Monge-Kantorovich formulation of Monge's optimal transport problem seeks to find an optimal coupling $\pi\in \mathcal{P}(\R^d \times \R^d)$ satisfying
\begin{align}\label{eqn:kant-l1}
\pi := \argmin_{\pi \in \Gamma(\mu, \nu)} \int \|x - y\| \ \mathrm{d}\pi(x, y),
\end{align}
where $\|\cdot\|$ denotes the Euclidean norm and $\Gamma(\mu, \nu) := \{\pi \in \mathcal{P}(\R^d \times \R^d) : (\operatorname{proj}_X)_{\#}\pi = \mu, \ (\operatorname{proj}_Y)_{\#}\pi = \nu \}.$  
Caffarelli, Feldman, and McCann \cite{mcann-caff-feldman}, along with Evans and Gangbo \cite{evans-gangbo}, constructed a solution to \eqref{eqn:kant-l1} in the form \( \pi = (\operatorname{Id} \otimes T)\# \mu \), where \(T\) is an optimal transport map that pushes \(\mu\) to \(\nu\) along straight-line paths, commonly referred to as transport rays connecting \(\mathcal{X}\) to \(\mathcal{Y}\) (see Section \ref{sec:l1-ot}). The entropic regularization of the above optimal transport problem is 
\begin{align}\label{defn:pi_eps}
\pi^{\varepsilon}:=\argmin_{\pi \in \Gamma(\mu, \nu)}\left(\int_{\mathbb{R}^d \times \mathbb{R}^d}\|x-y\| \mathrm{d} \pi+\varepsilon \mathsf{H}(\pi \mid \mu \otimes \nu)\right),    
\end{align}
where $\mathsf{H}$ represents the relative entropy defined by
\begin{align}\label{eqn:entropy}
\mathsf{H}(\pi \mid P):= \begin{cases}\int_{\mathbb{R}^d \times \mathbb{R}^d} \log \left(\frac{\mathrm{d} \pi}{\mathrm{~d} P}\right) \mathrm{d} \pi, & \text { if } \pi \ll P \\ \infty, & \text { otherwise }\end{cases}.
\end{align}
The variational problem described above is equivalent to stating that \( \pi^{\varepsilon} \) minimizes the relative entropy \( \mathsf{H}(\pi \mid e^{-\|x-y\|/\varepsilon} \mu \otimes \nu) \) among all couplings \( \pi \in \Gamma(\mu,\nu) \).
As $\varepsilon \to 0$, the entropically regularized coupling $\pi^\varepsilon$ weakly converges to a coupling $\Tilde{\pi}_{\text{opt}}$ that minimizes \eqref{eqn:kant-l1}. Since the cost $c(x, y):= \|x - y\|$ is not strictly convex, the Monge-Kantorovich problem generally admits infinitely many minimizers (cf. Remark 2.7.4 in \cite{figalli-glaudo}), leaving the selection of $\Tilde{\pi}_{\text{opt}}$ unclear as $\varepsilon$ vanishes. While Marino and Louet \cite{Marino} addressed this issue when $d=1$, the higher-dimensional case has remained an important open problem. In this paper, we resolve this question by providing a variational characterization of the limiting coupling $\Tilde{\pi}_{\text{opt}}$ when $d\geq 2.$

To state our main theorem, we introduce some notation. Let $\mathcal{X}, \mathcal{Y} \subset \mathbb{R}^d$ be compact sets, and let $\mu$ and $\nu$ be probability measures with continuous densities $p: \mathcal{X} \to \mathbb{R}_{>0}$ and $q: \mathcal{Y} \to \mathbb{R}_{>0}$ respectively. Let $\mathcal{R}$ be a \emph{maximal} transport ray, and denote by $\mu_{\llcorner \mathcal{R}}$ and $\nu_{\llcorner \mathcal{R}}$ the restrictions of $\mu$ and $\nu$ to $\mathcal{R} \cap \mathcal{X}$ and $\mathcal{R} \cap \mathcal{Y}$, respectively (see Section~\ref{sec:Restriction} for more details about their definitions). The limiting coupling $\Tilde{\pi}_{\text{opt}}$ is then characterized as follows:
\begin{theorem}\label{thm:Theorem1}
Let $d\geq 2$ and $\pi^\varepsilon \rightharpoonup \Tilde{\pi}^{\text{opt}}$ be the weak limit of the optimal entropic coupling as $\varepsilon \to 0.$ Then $\Tilde{\pi}^{\text{opt}}$ can be characterized as the minimizer of the following cost functional
\begin{align}\label{eq:Characterization}
    \Tilde{\pi}^{\text{opt}}_{\llcorner \mathcal{R}} = \operatorname{argmin} \left\{ \mathsf{H}\big(\pi|\mathcal{F}(x,y) \mu_{\llcorner \mathcal{R}}\otimes \nu_{\llcorner \mathcal{R}}\big):  \pi\in \Gamma(\mu_{\llcorner \mathcal{R}},\nu_{\llcorner \mathcal{R}}) \right\}
\end{align}
for any transport ray $\mathcal{R}$ and $\mathcal{F}: \mathcal{R} \cap \mathcal{X}\times \mathcal{R} \cap \mathcal{Y} \to \mathbb{R}_{\geq 0}$ is defined as  $$\mathcal{F}(x,y) :=  (2\pi \|x-y\|)^{\frac{d-1}{2}}e^{c\|x-y\|} $$   
where $c\geq 0$ is a constant which depends on $\mu$, $\nu$ and the transport ray $\mathcal{R}$.
\end{theorem}
\begin{remark}
     Since the set of couplings \( \Gamma(\mu, \nu) \) is compact in the weak topology, and any limiting coupling of the sequence \( \pi^{\varepsilon} \) is uniquely characterized by \eqref{eq:Characterization}, Theorem~\ref{thm:Theorem1} establishes that \( \pi^{\varepsilon} \) indeed converges to the limiting coupling \( \Tilde{\pi}^{\text{opt}} \) as described by \eqref{eq:Characterization}. It is important to note that Theorem~\ref{thm:Theorem1} does not explicitly specify the value of the constant \( c \) in \( \mathcal{F}(x, y) \). This constant depends implicitly on the limiting behavior of the derivatives of the entropic optimal transport potentials. Further details regarding this dependency will be provided in Section~\ref{subsec:proof-sketch}.
\end{remark}

\subsection{Proof Sketch}\label{subsec:proof-sketch}
The primary objective of this section is to demonstrate how the entropic optimal transport approach selects a unique optimal transport plan as the regularization parameter $\varepsilon \to 0$. For norm cost $\|x-y\|$, which is not strictly convex, the corresponding optimal transport map is not unique. However, \cite{mcann-caff-feldman} constructed a potential Monge map, which transports mass along transport rays. These transport rays were first introduced by \cite{evans1999differential} and are defined as line segments connecting points in the supporting domains of the marginal measures. Along these rays, the Kantorovich potentials decrease linearly. 

The work in \cite{evans1999differential} shows that the union of all transport rays contains the union of the supporting domains, and further, these rays are disjoint (see Section~\ref{sec:l1-ot}). Moreover, optimal transport plans with respect to the norm cost are supported on these transport rays (see Lemma~\ref{lem:support-lemma} for proof). Thus, to establish Theorem~\ref{thm:Theorem1}, it suffices to investigate how $\pi^\varepsilon$, the entropic optimal transport plan, aligns with the transport rays as $\varepsilon \to 0$. 
For a transport ray $\mathcal{R}$, consider cylinders $\mathfrak{C}\subset \mathcal{X}$ and $\mathfrak{C}'\subset \mathcal{Y}$ of radius $\omega\ll 1$ around the segments $\vec{aa'}\subset \mathcal{R}\cap \mathcal{X}$ and $\vec{dd'}\subset \mathcal{R}\cap \mathcal{Y}$ respectively. We show that 
 \begin{align}
     \pi^{\varepsilon}(\mathfrak{C}\times \mathfrak{C}')\sim \omega^{d-1} \int_{\vec{aa'}\times \vec{dd'}} d\tilde{\pi}^{\mathrm{opt}}_{\llcorner \mathcal{R}}(x,y), \qquad \text{as }\varepsilon \to 0. \label{eq:pi_pesilon_limit}
 \end{align}
Here $\tilde{\pi}^{\mathrm{opt}}_{\llcorner \mathcal{R}}$ is same as in Theorem~\ref{thm:Theorem1}. Now the final result follows by the weak convergence of $\pi^{\varepsilon}$ to $\tilde{\pi}^{\mathrm{opt}}$ as $\epsilon\to 0$ and further, letting $\omega \to 0$. Now we explain how we obtain the above approximation as $\varepsilon \to 0$. Since $\tilde{\pi}^{\mathrm{opt}}_{\llcorner \mathcal{R}}$ solves an entropic optimal transport problem between $\mu_{\llcorner \mathcal{R}}$ and $ \nu_{\llcorner \mathcal{R}}$ as revealed in \eqref{eq:Characterization}, we know 
$$\frac{d\tilde{\pi}^{\mathrm{opt}}_{\llcorner \mathcal{R}}}{d(\mu_{\llcorner \mathcal{R}}\otimes \nu_{\llcorner \mathcal{R}})}(x,y)  = (2\pi \|x-y\|)^{\frac{d-1}{2}}e^{c\|x-y\|}e^{\mathfrak{f}(x)+ \mathfrak{g}(y)}$$
 where $\mathfrak{f}:\mathcal{R}\cap\mathcal{X}\to \mathbb{R}$ and $\mathfrak{g}:\mathcal{R}\cap\mathcal{X}\to \mathbb{R}$ are the potential functions of the entropic optimal transport problem in \eqref{eq:Characterization}. We explain here how we obtain the form of the Radon-Nikodym derivative in the above display from the density of $\pi^{\varepsilon}$ with respect to $\mu\otimes \nu$, i.e.,  $$\frac{d\pi^{\varepsilon}}{d(\mu\otimes \nu)}(x,y) = \exp\Big(-\frac{\|x-y\|}{\varepsilon} + \frac{(f_{\varepsilon}(x) + g_{\varepsilon}(y))}{\varepsilon}\Big).$$ We define $\tilde{I}(x,y) := \|x-y\| - u(x_{\llcorner \mathcal{R}}) + u(y_{\llcorner \mathcal{R}})$ for all $(x,y)\in \mathcal{X}\times \mathcal{Y}$, $u$ is the Kantorovich potential of the optimal transport plan between $\mu$ and $\nu$, $(f_{\varepsilon}, g_{\varepsilon})$ are entropic optimal transport potential and $(x_{\llcorner \mathcal{R}},y_{\llcorner \mathcal{R}})$ are the orthogonal projection of $(x, y)$ onto $\mathcal{R}$. With $I(x,y)$, we rewrite $$\|x-y\| - (f_{\varepsilon}(x) + g_{\varepsilon}(y)) =I(x,y)-(f_{\varepsilon}(x)- u(x_{\llcorner \mathcal{R}}) + g_{\varepsilon}(y) + u(y_{\llcorner \mathcal{R}})).$$ 
 Along the transport ray $\mathcal{R}$,  $I(x,y)$ vanishes for any $(x,y) \in (\mathcal{R}\cap \mathcal{X})\times (\mathcal{R}\times \mathcal{Y})$. Furthermore, for any $x+\sqrt{\varepsilon}x^{\perp} \in \mathcal{X}$ and $y+ \sqrt{\varepsilon}y^{\perp} \in \mathcal{Y}$ with $(x,y)\in \mathcal{R}$ and $x^{\perp}, y^{\perp}$ are orthogonal to the direction of the transport ray $\mathcal{R}$, $I(\cdot,\cdot)$ expands as 
 $$ I(x+\sqrt{\varepsilon}x^{\perp}, y+ \sqrt{\varepsilon}y^{\perp}) \sim \varepsilon(1+o(1))\frac{\|J(x)x^{\perp} - J(y)y^{\perp}\|^2}{\|x-y\|}$$
where $J(x)$ is the Jacobian matrix of transformation of local coordinates around transport rays. A detailed discussion of this expansion can be found in Lemma~\ref{lemma:Main}. As $\varepsilon \to 0$, the entropic optimal transport potentials expand as 
\begin{align*}
\frac{f_{\varepsilon} (x+\sqrt{\varepsilon}x^{\perp}) - u(x)}{\varepsilon} &= \mathfrak{f}(x)+\frac{1}{\sqrt{\varepsilon}}(1+o(1)) \langle v_{\varepsilon}, J(x) x^{\perp} \rangle, \\
\frac{g_{\varepsilon}(x+\sqrt{\varepsilon}x^{\perp}) + u(x)}{\varepsilon} &= \mathfrak{g}(y)-\frac{1}{\sqrt{\varepsilon}}(1+o(1)) \langle v_{\varepsilon}, J(y) y^{\perp} \rangle 
\end{align*}
for some $v_{\varepsilon} \in \mathbb{R}^d$. Combining these expansion of $I(\cdot, \cdot)$ and $(f_{\varepsilon}, g_{\varepsilon})$ yields 
\begin{align}
\frac{d\pi^{\varepsilon}}{d(\mu\otimes \nu)}(x,y) &= \exp\Big(-(1+o(1))\frac{\|J(x)x^{\perp} - J(y)y^{\perp}\|^2}{\|x-y\|} \\&+\frac{1}{\sqrt{\varepsilon}}(1+o(1)) \langle v, J(x) x^{\perp} - J(y) y^{\perp}\rangle \Big) \exp(\mathfrak{f}(x)+ \mathfrak{g}(y)).
\end{align}
Furthermore, as $\varepsilon \to 0$, $v_{\varepsilon}/\sqrt{\varepsilon}$ converges. See Theorem~\ref{lem:Deri_Conv} for its proof. Notice that the above reductions show that the Radon-Nikodym derivative of $\pi^{\varepsilon}$ with respect to $\mu\otimes \nu$ locally behaves as a multivariate Gaussian distribution with respect to $x^{\perp}$ and $y^{\perp}$. Integrating this density around a local neighborhood of $\mathcal{R}$ of width $\omega$ shows
$$ \pi^{\varepsilon}(\mathfrak{C}\times \mathfrak{C}')\sim \omega^{d-1} \int_{\vec{aa'}\times \vec{dd'}}(2\pi \|x-y\|)^{\frac{d-1}{2}}e^{c\|x-y\|}e^{\mathfrak{f}(x)+ \mathfrak{g}(y)} d(\mu_{\llcorner \mathcal{R}}\otimes \nu_{\llcorner \mathcal{R}})(x,y).$$
Here, the constant $c$ is equal to $\lim_{\varepsilon\to 0} \|v_{\varepsilon}\|^2/\varepsilon.$ The proof of the above display is primarily contained in Lemma~\ref{lemma:Main}, where a more general result is established. The proof of Lemma~\ref{lemma:Main} relies on a multi-scale argument, distinguishing between contributions from regions near the transport ray and those farther away. Specifically, the contribution from a close neighborhood of the transport ray is analyzed using the expansion of the entropic optimal transport potential, as explained earlier. In contrast, the contribution from regions farther from the transport ray is estimated using Lemma~\ref{claim:lower bound}, which serves as a critical component of the proof of Lemma~\ref{lemma:Main}.
By proving the above display via Lemma~\ref{lemma:Main}, we establish the validity of \eqref{eq:pi_pesilon_limit}, thereby completing the proof of Theorem~\ref{thm:Theorem1}.

\subsection{Background}
The Monge-Kantorovich problem arises as a relaxation of the Monge problem, first proposed by Monge in 1781 \cite{Monge}. Given probability measures \(\mu, \nu \in \mathcal{P}(\R^d)\), the goal of the Monge problem is to find an optimal transport map \(T\) minimizing
\begin{align}\label{eqn:monge-problem}
 \inf \int_{\R^d} \|x - T(x)\| \, \mathrm{d}\mu(x),
\end{align}
under the constraint that \(T\) preserves mass, i.e., \(T_{\#}\mu = \nu\). 
While simple examples (e.g., when \(\mu\) is a Dirac measure and \(\nu\) is the standard Gaussian) show that \eqref{eqn:monge-problem} does not always admit a minimizer, Monge's insight introduced the concept of transport rays, which became central to subsequent developments. Sudakov \cite{Sudakov} was the first to attempt solving the Monge problem by decomposing \(\R^d\) into affine regions where the transport problem could be solved explicitly, and then gluing these solutions to construct an optimal transport map. However, this approach relied on the incorrect assumption that the absolute continuity of \(\mu\) with respect to the Lebesgue measure $\mathcal{L}^d$ implied the absolute continuity of the measures induced by the decomposition (cf. Section 6 in \cite{lecture-notes-ambrosio}). Several years later, Evans and Gangbo \cite{evans-gangbo} introduced new analytical tools and resolved the Monge problem under the assumptions \(\mu, \nu \ll \mathcal{L}^d\) with Lipschitz densities and \(\operatorname{supp} \mu \cap \operatorname{supp} \nu = \emptyset\). Subsequently, Caffarelli--Feldman--McCann \cite{mcann-caff-feldman} and Trudinger--Wang \cite{trudinger-wang} independently resolved the Monge problem for absolutely continuous measures with compact support in \(\R^d\).  Ambrosio \cite{lecture-notes-ambrosio} later established the existence of an optimal transport map under the very general assumption \(\mu \ll \mathcal{L}^d\) without requiring additional regularity assumptions. Furthermore, Feldman and McCann \cite{manifold} extended the results of \cite{mcann-caff-feldman} and \cite{trudinger-wang} to the Riemannian setting. On the other hand, when the cost function is squared distance, the seminal works of Brenier \cite{brenier} and McCann \cite{McCann} provide a complete characterization of the transport map as the gradient of a convex potential. Later, this result was generalized for any strictly convex cost function by Gangbo and McCann \cite{Gangbo&McCann}.


In recent years, optimal transport has garnered significant attention for its computational applications in fields such as statistics, machine learning, and economics. For a comprehensive overview, see \cite{peyré2020computationaloptimaltransport}. Among various computational approaches, the entropic regularization method popularized by Cuturi \cite{Cuturi} has become a particularly effective tool for approximating optimal transport plans efficiently, especially in large-scale settings. By incorporating a relative entropy term into the classical Kantorovich formulation, this method reformulates the problem into a computationally tractable form. It enables the use of scalable iterative algorithms, such as Sinkhorn’s algorithm, which are widely adopted due to their robustness and efficiency in high-dimensional applications. The statistical analysis of entropic optimal transport and its applications in machine learning and statistical inference has emerged as a vibrant research direction, leading to a broad range of developments—\cite{pooladian2021,gonzalez2022,divol2024}, to name a few.

Beyond its computational applications, entropic optimal transport has been the focus of significant theoretical developments. While it is not possible to list all contributions in this rapidly growing field, we highlight a few results that are particularly relevant to our work. Gamma-convergence techniques have demonstrated that weak limits of entropic optimizers correspond to solutions of the classical Monge-Kantorovich problem \cite{Christian, carlier2017convergence}. This result has been extended to more general settings with non-integrable cost function via $(c,\varepsilon)$-cyclical invariance in \cite{entropic-cyc-ld} where large deviation principles have been used to quantify convergence rates. Furthermore, asymptotic expansions of the entropic cost as \(\varepsilon \to 0\) have been explored in \cite{pal2024difference,conforti2021formula}, while \cite{potentials} examined the convergence of Schrödinger potentials in Polish spaces. Regarding stability, \cite{stability} established the robustness of entropically regularized solutions under variations in both the marginals and the cost function. We refer to \cite{marcel-notes, Chewi2024statisticalOT} for more details.


Despite significant progress, understanding the precise behavior of entropically regularized couplings as \(\varepsilon \to 0\) remains a challenging open problem, particularly when the cost function lacks strict convexity. This question has been partially addressed in the discrete setting by \cite[Proposition~4.1]{peyré2020computationaloptimaltransport}, where the authors demonstrated that the limit of the entropic optimal transport maximizes entropy among all optimal transport couplings.  

In the continuous setting, this question has been resolved only in one dimension with the linear cost function \(c(x,y) = |x-y|\). Marino and Louet \cite{Marino} provided a unique characterization of the limiting optimal transport coupling as the solution to a variational problem. However, extending this result to higher dimensions \(d > 1\) has remained an open problem. In this work, we address this gap by establishing a variational 
characterization of the limiting coupling in
our main result, Theorem~\ref{thm:Theorem1}.  

 Entropic optimal transport is also known as the (static) Schr\"odinger bridge problem. The Schrödinger problem explores the most likely evolution of a particle system in thermal equilibrium, given a reference law \( R \) for its configuration at initial and final times \( t_0 \) and \( t_1 \), and observed marginals \( (\mu, \nu) \) that differ from those of \( R \). The solution minimizes the relative entropy \( H(\cdot | R) \) over all couplings \( \Gamma(\mu, \nu) \), providing the optimal joint law \( \pi^* \). The minimization of \( H(\cdot | R) \) over \( \Gamma(\mu, \nu) \) aligns with the entropic optimal transport problem when the cost function \( c \) is defined as \( c := -\varepsilon \log(\alpha^{-1} dR/d(\mu \otimes \nu)) \), where \( \varepsilon > 0 \) and \( \alpha \) is a normalizing constant. 
 While this is the static formulation, the dynamic Schr\"odinger bridge problem extends this framework to model the continuous evolution of the particle system over \( t \in [t_0, t_1] \), with the static problem serving as a projection to the marginals. For detailed surveys and historical context see \cite{marcel-notes}.

A large deviations principle arises in the small-noise or small-time limit, where the reference measure \( R(\varepsilon) \) collapses to a deterministic coupling, mapping a particle at \( x \) to a fixed destination \( f(x) \) \cite{entropic-cyc-ld}. Theorem~\ref{thm:Theorem1} extends this analysis by characterizing the small-noise limit of the static Schrödinger bridge problem when the cost function \( c \) corresponds to the distance cost. Investigating whether this principle can be generalized to the dynamic Schrödinger bridge problem on the path space remains an intriguing open question.


A fundamental question in the study of any regularization scheme in the calculus of variations is understanding the behavior of minimizers as the regularization parameter tends to zero. This question resonates across many areas of mathematics, where regularization often plays a crucial role in selecting meaningful solutions. In fluid dynamics, the inviscid limit problem (cf. \cite{vicol}) investigates the asymptotic behavior of solutions to the Navier-Stokes equations as the viscosity parameter approaches zero. In differential geometry, the seminal work of Sacks and Uhlenbeck \cite{Sacks} introduced a perturbed Dirichlet energy to construct minimal immersed spheres, illustrating how regularization schemes can lead to geometrically significant minimizers. Similarly, in game theory, recent works \cite{Hajek2019NonuniqueMFG, bayraktar2020non} have shown that in cases where mean field games exhibit non-unique Nash equilibria, the equilibria that arise as limits of finite $N$ Nash equilibria, as $N$ tends to infinity, correspond precisely to the entropy solution. See also \cite{Carlier-Blanchet} for the connection between Cournot-Nash equilibria and entropic optimal transport. Building on this foundation, our work contributes to the broader research landscape by providing a rigorous variational characterization of the limiting coupling \(\Tilde{\pi}_{\text{opt}}\) as \(\varepsilon \to 0\). This addresses fundamental theoretical questions within the framework of the entropic Monge-Kantorovich problem, advancing the understanding of how entropic regularization selects specific optimal transport plans.




\subsection{Structure of the Paper}
The remainder of this paper is organized as follows. In Section \ref{sec:preliminaries}, we provide the necessary preliminaries, including the formulation of the entropic optimal transport problem and its connection to the classical Monge-Kantorovich problem. We also review key results on transport rays and define the restriction of measures to these rays. In Section \ref{sec:proof-main}, we prove Theorem \ref{thm:Theorem1}, by establishing a variational characterization of the limiting coupling $\Tilde{\pi}_{\text{opt}}$ as the entropic regularization parameter $\varepsilon \to 0$. Finally, Section \ref{sec:eot} examines the asymptotic behavior of entropic potentials around transport rays. Through precise expansions, we derive their limiting form, which plays a crucial role in the proof of our main result.

\section*{Acknowledgments}
The authors thank Soumik Pal for his thoughtful feedback on an earlier draft of this manuscript. SA gratefully acknowledges insightful discussions with Luigi Ambrosio regarding this problem during a 2023 conference at ETH Zürich held in his honor. PG's research was partially supported by NSF grant DMS-2515172 (originally DMS-2153661). PG is also grateful to Marcel Nutz for informing the authors, following the initial ArXiv posting of this paper, that he and his collaborator Chenyang Zhong were independently finalizing a manuscript containing similar results as us and for bringing to PG's attention a lecture video from a few years ago addressing a related problem.

\section{Preliminaries}\label{sec:preliminaries}
Recall that \( \mathcal{X}, \mathcal{Y} \subset \mathbb{R}^d \), and let \( \mathcal{P}(\mathcal{X}) \) denote the space of all probability measures supported on \( \mathcal{X} \). Additionally, recall that \( \mu \in \mathcal{P}(\mathcal{X}) \) and \( \nu \in \mathcal{P}(\mathcal{Y}) \). We now outline the following assumptions on the measure spaces \((\mathcal{X}, \mu)\) and \((\mathcal{Y}, \nu)\).
\begin{enumerate}
    \item $\mathcal{X}$ and $\mathcal{Y}$ are compact subsets of $\mathbb{R}^d$ that are quantitatively disjoint, i.e. $\mathcal{X}\cap \mathcal{Y} = \emptyset$ and $\operatorname{dist}(\mathcal{X},\mathcal{Y})>c_0$ for some positive constant $c_0>0.$
    \item $\mu(\mathrm{d}x) =p(x)\mathrm{d}x$ and $\nu(\mathrm{d}y) = q(y)\mathrm{d}y$ where $p$ and $q$ are smooth densities bounded away from $0$ and $+\infty$. 
\end{enumerate}
\subsection{Dual Formulation}
Let $v_0$ be the optimal transport cost function in \eqref{eqn:kant-l1}. 
Given $\varepsilon>0$, let $v_{\varepsilon}$ be the entropic optimal transport cost function in \eqref{defn:pi_eps}.
Alternatively, $v_0$ and $v_\varepsilon$ can be characterized using a dual formulation
\begin{align}\label{defn:dual-v0}
v_0 &:=\sup _{\substack{f \in \operatorname{Lip}_1\left(\mathcal{X}\right) \\ g \in \operatorname{Lip}_1\left(\mathcal{Y}\right)}}\left\{\int_{\mathcal{X}} f \ \mathrm{d} \mu+\int_{\mathcal{Y}} g \ \mathrm{d} \nu: f(x)+g(y) \leq \|x-y\|, \forall (x,y)\in \mathcal{X}\times \mathcal{Y}\right\},
\end{align}
and 
\begin{align}\label{defn:dual-v_eps}
v_\varepsilon := \sup _{\substack{f_{\varepsilon} \in L^1_\mu\left(\mathcal{X}\right) \\ g_{\varepsilon} \in L^1_\nu\left(\mathcal{Y}\right)}}\left\{\int_{\mathcal{X}} f_{\varepsilon} \ \mathrm{d} \mu +\int_{\mathcal{Y}} g_{\varepsilon} \ \mathrm{d} \nu-\varepsilon \log \int_{\mathcal{X} \times \mathcal{Y}} e^{\frac{f_{\varepsilon}(x)+g_{\varepsilon}(y)-\|x-y\|}{\varepsilon}} \ \mathrm{d} \mu(x) \mathrm{d} \nu(y)\right\}
\end{align}
where $\operatorname{Lip}_1(\mathcal{X})$ denotes the space of all real-valued Lipschitz functions on $\mathcal{X}$ with Lipschitz constant $\leq 1$ and $L^1_\mu(\mathcal{X})$ denotes the space of real-valued integrable functions with respect to the measure $\mu$ on $\mathcal{X}.$ Any pair  $(f_K,g_K)$ minimizing \eqref{defn:dual-v0} is called a Kantorovich potential whereas the pair $(f_{\varepsilon}, g_{\varepsilon})$ minimizing \eqref{defn:dual-v_eps} is called an entropic optimal transport (EOT) potential. 
\subsection{Transport Rays}\label{sec:l1-ot}
We present some relevant facts about transport rays following \cite{mcann-caff-feldman}. There exists $u\in \operatorname{Lip}_1(\R^d)$ such that the Kantorovich potential $(f_K,g_K)$ in \eqref{defn:dual-v0} can be written as
\begin{align}
    u(x)=-f_K(x)\text{ on }\mathcal{X},\quad u(y)=g_K(y) \text{ on }\mathcal{Y}.
\end{align}
 A transport ray $\mathcal{R}$ is a segment with endpoints $a$, $b \in \mathbb{R}^d$ such that
\begin{enumerate}
     \item $a \in \mathcal{X}, b \in \mathcal{Y}, a \neq b$,
     \item $u(a)-u(b)=\|a-b\|$, and 
     \item for any $t>0$ the ray $a_t:=a+t(a-b) \in \mathcal{X}$ satisfies $|u(a_t)-u(b)|<\|a_t-b\|$ and $b_t = b+t(b-a)\in \mathcal{Y}$ satisfies $|u(b_t)-u(a)|< \|b_t-a\|.$
 \end{enumerate}
The points $a$ and $b$ are called the upper and lower ends of $\mathcal{R}$, respectively. Let $T_1$ be the set of all points that reside on transport rays and define its complementary set $T_0$, often called the rays of length zero, by
$$
T_0:=\left\{z \in \mathcal{X} \cap \mathcal{Y}:\left|u(z)-u\left(z^{\prime}\right)\right|<\left\|z-z^{\prime}\right\| \text { for any } z^{\prime} \in \mathcal{X} \cup \mathcal{Y}, z^{\prime} \neq z\right\} .
$$
 Lemma 9 of \cite{mcann-caff-feldman} show that $\mathcal{X} \cup \mathcal{Y} \subseteq T_0 \cup T_1$. Although this result can be derived from \cite{mcann-caff-feldman}, we explicitly demonstrate in Lemma~\ref{lem:support-lemma} of Section~\ref{lem:support-lemma} that the Monge-Kantorovich optimal transport for \eqref{eqn:kant-l1} is supported on the set \(T_1\).

A key tool in our analysis is the existence of a Lipschitz change of coordinates on the covering of the set \(T_1\), as established in Lemma 22 of \cite{mcann-caff-feldman}. To proceed, we first introduce the necessary notation required to state this result, followed by the formal statement itself. 
For any $p \in \mathbb{R}$, let $S_p$ be the level set $\left\{x \in \mathbb{R}^d \mid u(x)=p\right\}$. Lemma 18 of \cite{mcann-caff-feldman} shows that the set
$$
S_p \cap\left\{x \in \mathbb{R}^d \mid u \text { is differentiable at } x \text { and } D u(x) \neq 0\right\}
$$
has a countable covering consisting of Borel sets $S_p^i \subset S_p$, such that for each $i \in \mathbb{N}$ there exist Lipschitz coordinates $U: \mathbb{R}^d \rightarrow \mathbb{R}^{d-1}$ and $V: \mathbb{R}^{d-1} \rightarrow \mathbb{R}^d$ satisfying
$$
V(U(x))=x \text { for all } x \in S_p^i .
$$
Now for any $p\in \mathbb{Q}, i, j \in \mathbb{N}$ define a ray cluster $T_{p i j}:=\bigcup R_z$ as the union of all transport rays $R_z$ which intersect $S_p^i$, and for which the point of intersection $z \in S_p^i$ is separated from both endpoints of the ray by distance greater than $1 / j$ in $\|\cdot\|$. The same cluster, but with ray ends omitted, is denoted by $T_{p i j}^0:=\bigcup_z\left(R_z^0\right)$. Lemma 20~\cite{mcann-caff-feldman} shows that the ray clusters $T_{p i j}$ define a countable covering of all transport rays $T_1$. Moreover, each $T_{p i j}$ and transport ray $R$ satisfy, either $(R)^0 \subset T_{p i j}$ or $(R)^0 \cap T_{p i j}=\emptyset$. Now we recall Lemma 22 of \cite{mcann-caff-feldman}.
\begin{lemma}[Lemma~22 of \cite{mcann-caff-feldman}]\label{lem:lip-change-of-variable}
Each ray cluster $T_{p i j} \subset \mathbb{R}^d$ admits coordinates $G=G_{p i j}: T_{p i j}^0 \rightarrow \mathbb{R}^{d-1} \times \mathbb{R}^1$ with inverse $F=F_{p i j}: G\left(T_{p i j}^0\right) \rightarrow \mathbb{R}^d$ satisfying:
\begin{enumerate}
    \item F extends to a Lipschitz mapping between $\mathbb{R}^{d-1} \times \mathbb{R}^1$ and $\mathbb{R}^d$.
    \item For each $\lambda>0, G$ is Lipschitz on $T_{p i j}^\lambda:=\left\{x \in T_{p i j}^0 \mid\|x-a\|,\|x-b\|>\lambda\right\}$, where $a$ and $b$ denote the endpoints of the (unique) transport ray $R_z$.
    \item $F(G(x))=x$ for each $x \in T_{p i j}^0$.
    \item If a transport ray $R_z \subset T_{p i j}$ intersects $S_p^i$ at $z$, then each interior point $x \in\left(R_z\right)^0$ of the ray satisfies
$$
G(x)=(U(z), u(x)-u(z)),
$$
where $U: \mathbb{R}^d \rightarrow \mathbb{R}^{d-1}$ gives the Lipschitz coordinates on $S_p^i$.
\end{enumerate}
\end{lemma}
We also need a result that controls the separation of transport rays arising from two points on the same level set. Let $\chi: \mathbb{R}^d \rightarrow \mathbb{R}^d$ be a ray direction which is specified as follows. If $z$ is an interior point of a transport ray $R$ with upper and lower endpoints $a, b$, then
$$
\chi(z):=\frac{a-b}{\|a-b\|} .
$$
Define $\chi(z)=0$ for any point $z \in \mathbb{R}^d$ not the interior point of a transport ray. 
\begin{lemma}[Lemma 16~\cite{mcann-caff-feldman}]\label{lem:lip-control}
Let $\mathcal{R}_1$ and $\mathcal{R}_2$ be transport rays, with upper end $a_k$ and lower end $b_k$ for $k=1,2$ respectively. If there are interior points $y_k \in\left(\mathcal{R}_k\right)^0$ where both rays pierce the same level set of Monge's potential $u\left(y_1\right)=u\left(y_2\right)$, then the ray directions satisfy a Lipschitz bound
$$
\left\|\chi\left(y_1\right)-\chi\left(y_2\right)\right\| \leq \frac{1}{\sigma}\left\|y_1-y_2\right\|,
$$
with $\sigma:=\min \left\{\left\|y_k-a_k\right\|,\left\|y_k-b_k\right\|\right\}$ for $k=1,2.$
\end{lemma}
\subsection{Restriction of Measures to Transport Rays}\label{sec:Restriction}
In this section, we formally define the restriction of a probability measure to some transport ray, henceforth denoted by $\mathcal{R}$ with endpoint $[a,b]$ for some $a,b\in \R^{d}.$
\begin{definition}[Blow-Up around a line]
 We define $\delta$-neighbourhood of $\mathcal{R}$ by 
\begin{align}
    B_\delta(\mathcal{R}) := \{z\in \R^{d}: \operatorname{dist}(z,R)\leq \delta\}.
\end{align}
\end{definition}
\begin{definition}[Measure Restricted to a Slice]\label{defn:slice-mu}
Denote the (interior) slice of $\mathcal{R}$ by $S^\mu_{\mathcal{R}}= \operatorname{spt} (\mu) \cap \mathcal{R}^\circ$. For $I\subset  S^\mu_{\mathcal{R}}$, we define the sliced probability measure $\mu_{\llcorner \mathcal{R}}\in \mathcal{P}(\R)$ as the following weak limit
\begin{align*}
  \mu_{\llcorner \mathcal{R}}(I):=\lim _{\delta \rightarrow 0} \frac{\mu\left(B_\delta(I)\right)}{\mu\left(B_\delta(S^\mu_\mathcal{R})\right)}.
\end{align*}
We define $\nu_{\llcorner \mathcal{R}}\in \mathcal{P}(\R)$ and $S^\nu_\mathcal{R}$ in a similar manner.
\end{definition}
We now generalize this definition to the coupling $\pi\in \Gamma(\mu,\nu).$
\begin{definition}\label{defn:slice-pi}
Let $U\times V \subset S^\mu_{\mathcal{R}}\times S^\nu_{\mathcal{R}}$ be a Borel subset of the transport ray $\mathcal{R}$, then  
    \begin{align*}
        \pi_{\llcorner \mathcal{R}}(U\times V) = \lim_{\delta\to 0} \frac{\pi(B_\delta(U)\times B_\delta(V))}{\pi(B_\delta(S^\mu_\mathcal{R})\times B_\delta(S^\nu_\mathcal{R}))}.
    \end{align*}
\end{definition}
\begin{lemma}[Lemma 3.2 in \cite{stability}]\label{lem:stability-entropic}
Let $\pi \in \Gamma(\mu,\nu)$. Consider for $x \in \operatorname{spt} \mu = \mathcal{X} $ and $r>0$. Define $\pi_x^{(r)} \in \mathcal{P}({\mathcal{Y}})$ as
$$
\pi_x^{(r)}(C)=\frac{\pi\left(B_r(x) \times C\right)}{\mu\left(B_r(x)\right)}, \quad C \in \mathcal{B}({\mathcal{Y}}),
$$
where $\mathcal{B}(\mathcal{Y})$ is the space of all Borel sets of $\mathcal{Y}.$ Then, there exists ${X}_0 \subset \mathcal{X}$ with $\mu\left({X}_0\right)=1$ such that for all $x \in {X}_0$, the weak limit
\begin{align}\label{defn:pi_x}
\pi_x:=\lim _{r \rightarrow 0} \pi_x^{(r)}    
\end{align}
exists. Moreover, $\pi_x$ defines a regular conditional probability of $\pi$ given $x$.
\end{lemma}

\section{Proof of Theorem~\ref{thm:Theorem1}}\label{sec:proof-main}
By Proposition~3.2 of \cite{entropic-cyc-ld}, as $\varepsilon \to 0$,  any subsequential limit of $\pi^{\varepsilon}$ is an optimal transport coupling. Suppose $\Tilde{\pi}^{\mathrm{opt}}$ is one such weak limit. We first need to show the \emph{restriction} of $\Tilde{\pi}^{\mathrm{opt}}$ on the transport ray, i.e.,   $\Tilde{\pi}^{\mathrm{opt}}_{\llcorner \mathcal{R}}$ is absolutely continuous with respect to the product measure $\mu_{\llcorner \mathcal{R}} \otimes \nu_{\llcorner \mathcal{R}}$. 
Furthermore, to complete the proof, we need to show that the Radon-Nikodym derivative of the sliced probability measure $\Tilde{\pi}^{\mathrm{opt}}_{\llcorner \mathcal{R}}$ with respect to $\mu_{\llcorner \mathcal{R}}\otimes \nu_{\llcorner \mathcal{R}}$ takes the following form 
\[\frac{d\Tilde{\pi}^{\text{opt}}_{\llcorner \mathcal{R}}}{d(\mu_{\llcorner \mathcal{R}} \otimes \nu_{\llcorner \mathcal{R}})}(x,y) = (2\pi \|x-y\|)^{\frac{d-1}{2}} e^{c\|x-y\|} e^{\mathfrak{f}(x)+ \mathfrak{g}(y)}, \qquad (x,y) \in  \mathcal{X}_{\llcorner \mathcal{R}}\times \mathcal{Y}_{\llcorner \mathcal{R}}\]
where $\mathfrak{f}: \mathcal{X}_{\llcorner \mathcal{R}}\to \mathbb{R}$ and $\mathfrak{g}: \mathcal{Y}_{\llcorner \mathcal{R}}\to \mathbb{R}$ are uniquely determined (up to some additive constants) by the following relations:
\begin{align}
    \mathfrak{f}(x) &+ \log \Big(\int_{S^\nu_\mathcal{R}}(2\pi \|x-y\|)^{\frac{d-1}{2}} e^{c\|x-y\|} e^{\mathfrak{g}(y)} \mathrm{d}\nu_{\llcorner \mathcal{R}}\Big) = 0, \label{eq:mathfrak_f}\\
   \mathfrak{g}(y) &+ \log \Big(\int_{S^\mu_\mathcal{R}}(2\pi \|x-y\|)^{\frac{d-1}{2}} e^{c\|x-y\|} e^{\mathfrak{f}(x)} \mathrm{d}\mu_{\llcorner \mathcal{R}}\Big) = 0. \label{eq:mathfrak_g}
\end{align}
We will prove this using the following lemma, but first, we introduce the necessary notation.

Recall that $\{S^{i}_p\}_i$ is the Borel cover of the level sets $S_p= \{x\in \mathbb{R}^{d}| u(x) = p\}$ and $T^0_{pij}$ be the cluster of transport rays which intersect $S^i_p$ at some point $z\in S^i_p$.  Consider a partial covering of $S^i_p \cap \mathcal{X}$ by disjoint intrinsic balls $\{B_{\delta}(x_k)\}_{k=1}^{\ell}$ where $x_k\in S^i_p$ is a point of intersection of some transport ray $\mathcal{R}_{x_k}$ and $S^i_p$. We can further write 
$Z: = (S_p^i \cap \mathcal{X})\Delta \bigcup_{k=1}^{\ell} B_\delta(x_k)$ for some $\ell = \ell(\delta)\in \mathbb{N}$ where $\mu_{\llcorner (S^{i}_p \cap \mathcal{X})}(Z) \to 0$ as $\delta \to 0$ and $x_k \in S^i_p$ for $k\geq 1$. For any transport ray $\mathcal{R}$, let $\mathfrak{S}^{\mu}_{\mathcal{R}}$ and $\mathfrak{S}^{\nu}_{\mathcal{R}}$ be the segments of $\mathcal{R}$ confined within $\mathcal{X}$ and $\mathcal{Y}$ respectively.

\begin{lemma}\label{lemma:Main}
 Fix a cluster of rays $T^0_{pij}$ for some $p \in \mathbb{Q}$, $i,j \in \mathbb{N}$. Set $\delta = \varepsilon^{\frac{1}{2} - \kappa}$ for some $\kappa\ll 1$. Let $O_{pij}$ be an open subset of $(T^0_{pij}\cap \mathcal{X}) \times (T^{0}_{pij}\cap \mathcal{Y}).$
Then, as $\varepsilon \to 0$, 
\begin{align*}
\lim_{\varepsilon\to 0}\left|\pi^{\varepsilon} (O_{pij}) - \sum_{k=1, B_{\delta}(x_k)\subset O_{pij}}^{\ell} \delta^{d -1} \int_{\mathfrak{S}^\mu_{\mathcal{R}_{x_k}}\times \mathfrak{S}^\nu_{\mathcal{R}_{x_k}} \cap O_{pij}} \Pi_{\mathcal{R}_{x_k}}(x,y) d(\mu_{\llcorner {\mathcal{R}_{x_k}}} \otimes \nu_{\llcorner {\mathcal{R}_{x_k}}})  \right|\to 0.
\end{align*}
where $\Pi_{\mathcal{R}_{x_k}}:  \mathfrak{S}^\mu_{\mathcal{R}_{x_k}} \times  \mathfrak{S}^\nu_{\mathcal{R}_{x_k}} \mapsto \mathbb{R}_{\geq 0}$ is defined  as 
$$\Pi_{\mathcal{R}_{x_k}}(x,y):= (2\pi \|x-y\|)^{\frac{d-1}{2}} e^{c\|x-y\|} e^{\mathfrak{f}(x) + \mathfrak{g}(y)} $$
with $(\mathfrak{f},\mathfrak{g})$ being uniquely determined by $\int \Pi_{\mathcal{R}_{x_k}}(x,y) \mathrm{d}\mu_{\llcorner {\mathcal{R}_{x_k}}} (x)= \nu_{\llcorner {\mathcal{R}_{x_k}}}(y)$ and $\int \Pi_{\mathcal{R}_{x_k}}(x,y) \mathrm{d}\nu_{\llcorner {\mathcal{R}_{x_k}}} (y)= \mu_{\llcorner {\mathcal{R}_{x_k}}}(x)$ up to some additive constants. 
\end{lemma}
We defer the proof of this lemma to Section~\ref{subsec:Lemma_Main_Proof} and proceed to complete the proof of Theorem~\ref{thm:Theorem1}.


 Fix a transport ray $\mathcal{R}\subset T^{0}_{pij}$. Recall that $\vec{aa'} \subset \mathfrak{S}^{\mu}_{\mathcal{R}}$ and $\vec{dd'} \subset \mathfrak{S}^{\nu}_{\mathcal{R}}$ be two line segments (or, slices) on $\mathcal{R}$. Let $\mathfrak{A}_{\omega}$ and $\mathfrak{D}_{\omega}$ be two cylinders around $\vec{aa'}$ and $\vec{dd'}$ respectively of radius $\omega$. Define $O_{pij}:=(\mathfrak{A}_{\omega}\cap T^0_{pij}\cap \mathcal{X})\times (\mathfrak{D}_{\omega}\cap T^0_{pij}\cap\mathcal{Y})$. By Lemma~\ref{lemma:Main},  we have 
 $$ \pi^{\varepsilon} (O_{pij}) = \sum_{k=1, B_{\delta}(x_k)\subset O_{pij}}^{\ell} \delta^{d -1} \int_{\mathfrak{S}^\mu_{\mathcal{R}_{x_k}}\times \mathfrak{S}^\nu_{\mathcal{R}_{x_k}} \cap O_{pij}} \Pi_{\mathcal{R}_{x_k}}(x,y) d(\mu_{\llcorner {\mathcal{R}_{x_k}}} \otimes \nu_{\llcorner {\mathcal{R}_{x_k}}}) +o_{\varepsilon}(1).$$
Note that $\mathfrak{f}_k$ and $\mathfrak{g}_k$ satisfy the same recursive relation as the dual potentials in the entropic optimal transport problem between $\mu_{\llcorner \mathcal{R}_{x_k}}$ and $\nu_{\llcorner\mathcal{R}_{x_k}}$ with respect to the cost function $-\frac{d-1}{2}\log(2\pi \|x-y\|) - c\|x-y\|$. By the stability of entropic optimal transport potential \cite{stability}, $(\mathfrak{f}_k,\mathfrak{g}_k)$ converges to $(\mathfrak{f}, \mathfrak{g})$ (defined via \eqref{eq:mathfrak_f} and \eqref{eq:mathfrak_g}) as $\omega \to 0$.  Thus, we have 
 \[\pi^{\varepsilon} (O_{pij}) = \omega^{d-1} (1+o_{\omega}(1))\int_{\vec{aa'}\times \vec{dd'}} \Pi_{\mathcal{R}}(x,y) d(\mu_{\llcorner \mathcal{R}} \otimes \nu_{\llcorner \mathcal{R}}) + o_{\varepsilon}(1).\]
By first letting $\varepsilon\to 0$ and then, dividing both sides by $\omega^{d-1}$ and sending $\omega \to 0$ yields 
$$ \Tilde{\pi}^{\text{opt}}(\vec{aa'}\times \vec{dd'}) = \int_{\vec{aa'}\times \vec{dd'}} \Pi_{\mathcal{R}}(x,y) d(\mu_{\llcorner \mathcal{R}} \otimes \nu_{\llcorner \mathcal{R}}).$$
This completes the proof.

\subsection{Proof of Lemma~\ref{lemma:Main}}\label{subsec:Lemma_Main_Proof}
Recall that $\delta = \varepsilon^{\frac{1}{2}-\kappa}$ for some $0<\kappa \ll 1$. For any $z$, $\mathcal{R}_z$ denotes the maximal transport ray passing through $z$ and $\mathcal{R}^{\circ}_z$ denotes the same transport with ray ends omitted. Let us define
\begin{align}
    A_{\delta}^k := \bigcup_{z\in B_\delta(x_k)}(\mathcal{R}^{\circ}_z\cap \mathsf{Proj}_{\mathcal{X}}(O_{pij})), &\qquad  
    D^k_{\delta} := 
    \bigcup_{z\in B_\delta(x_k)}(\mathcal{R}^{\circ}_z\cap \mathsf{Proj}_{\mathcal{Y}}(O_{pij})),  \\ 
    D^k_{\delta^{1-\kappa}} := 
    \bigcup_{z\in B_{\delta^{1-\kappa}}(x_k)} &(\mathcal{R}^{\circ}_z\cap \mathsf{Proj}_{\mathcal{Y}}(O_{pij}))
\end{align}
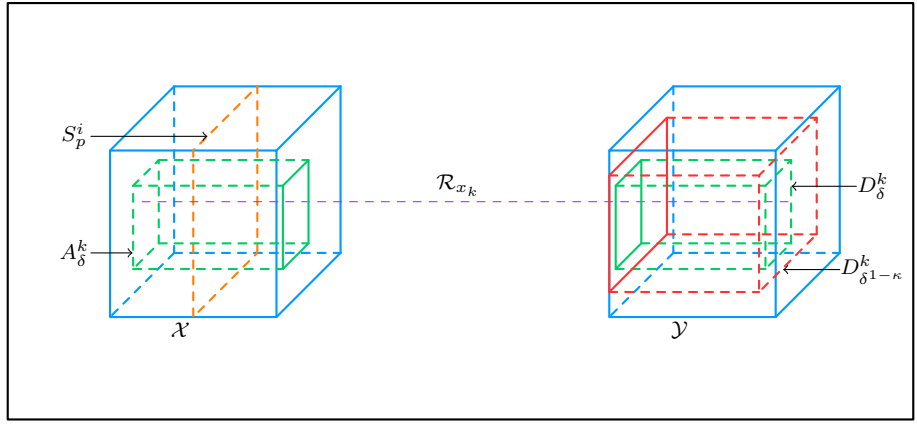
\begin{figure}
\begin{tikzpicture}[scale=2.2, line join=round, line cap=round]
\definecolor{mywhite}{rgb}{0,0,0}
\definecolor{myblue}{rgb}{0,0.6,1}
\definecolor{mygreen}{rgb}{0,0.8,0.4}
\definecolor{myorange}{rgb}{1,0.5,0}
\definecolor{myred}{rgb}{1,0.2,0.2}
\definecolor{myviolet}{rgb}{0.6,0.2,1}

\fill[white] (-1,-1) -- (-1, 0) -- (4, 0) -- (4, -1) -- cycle; 
\fill[white] (-1,1) -- (-1, 1.5) -- (4, 1.5) -- (4, 1) -- cycle; 

\draw[myblue, thick,dashed] (0,0,0) -- (1,0,0);
\draw[myblue, thick] (1,0,0)-- (1,1,0);
\draw[myblue, thick] (1,1,0)-- (0,1,0);
\draw[myblue, thick, dashed] (0,1,0) -- (0,0,0);

\draw[myblue, thick] (0,0,1) -- (1,0,1) -- (1,1,1) -- (0,1,1) -- cycle;

\draw[myblue, thick, dashed] (0,0,0) -- (0,0,1);
\draw[myblue, thick] (1,0,0) -- (1,0,1);
\draw[myblue, thick] (1,1,0) -- (1,1,1); 
\draw[myblue, thick] (0,1,0) -- (0,1,1);

\node at (0.5, 0, 1.2) {$\mathcal{X}$};

\draw[myblue, thick,dashed] (3,0,0) -- (4,0,0);
\draw[myblue, thick] (4,0,0)-- (4,1,0);
\draw[myblue, thick] (4,1,0)-- (3,1,0);
\draw[myblue, thick, dashed] (3,1,0) -- (3,0,0);

\draw[myblue, thick] (3,0,1) -- (4,0,1) -- (4,1,1) -- (3,1,1) -- cycle;

\draw[myblue, thick, dashed] (3,0,0) -- (3,0,1);
\draw[myblue, thick] (4,0,0) -- (4,0,1);
\draw[myblue, thick] (4,1,0) -- (4,1,1); 
\draw[myblue, thick] (3,1,0) -- (3,1,1);

\node at (3.5, 0, 1.2) {$\mathcal{Y}$};

\draw[myorange, thick, dashed] (0.5,0,0) -- (0.5,1,0) -- (0.5,1,1) -- (0.5,0,1) -- cycle;

\node at (-0.6,0.7) {$S^i_p$};
\draw[->] (-0.5,0.7) -- (0.2,0.7);

\draw[myviolet, thin,dashed] (0,0.5,0.5) -- (3.9,0.5,0.5);

\node at (1.9,0.6,0.5) {$\mathcal{R}_{x_k}$};

\draw[mygreen, thick, dashed] 
(0.1,0.25,0.5) -- (0.1,0.75,0.5) -- (0.1,0.75,0.9) -- (0.1,0.25,0.9) -- cycle;

\draw[mygreen, thick] 
(1,0.25,0.5) -- (1,0.75,0.5) -- (1,0.75,0.9) -- (1,0.25,0.9) -- cycle;

\node at (-0.6,0) {$A^k_{\delta}$};
\draw[->] (-0.5,0) -- (-0.25,0);

\draw[mygreen, thick, dashed] (0.1,0.25,0.5) -- (1,0.25,0.5);
\draw[mygreen, thick, dashed] (0.1,0.75,0.5) -- (1,0.75,0.5);
\draw[mygreen, thick, dashed] (0.1,0.75,0.9) -- (1,0.75,0.9);
\draw[mygreen, thick, dashed] (0.1,0.25,0.9) -- (1,0.25,0.9);

\draw[mygreen, thick] 
(3,0.25,0.5) -- (3,0.75,0.5) -- (3,0.75,0.9) -- (3,0.25,0.9) -- cycle;
\draw[mygreen, thick, dashed] 
(3.9,0.25,0.5) -- (3.9,0.75,0.5) -- (3.9,0.75,0.9) -- (3.9,0.25,0.9) -- cycle;

\draw[mygreen, thick, dashed] (3,0.25,0.5) -- (3.9,0.25,0.5);
\draw[mygreen, thick, dashed] (3,0.75,0.5) -- (3.9,0.75,0.5);
\draw[mygreen, thick, dashed] (3,0.75,0.9) -- (3.9,0.75,0.9);
\draw[mygreen, thick, dashed] (3,0.25,0.9) -- (3.9,0.25,0.9);

\draw[myred, thick] 
(3,0.15,0.1) -- (3,0.85,0.1) -- (3,0.85,1) -- (3,0.15,1) -- cycle;

\draw[myred, thick,dashed] 
(3.9,0.15,0.1) -- (3.9,0.85,0.1) -- (3.9,0.85,1) -- (3.9,0.15,1) -- cycle;

\draw[myred, thick,dashed] (3,0.15,0.1) -- (3.9,0.15,0.1);
\draw[myred, thick,dashed] (3,0.85,0.1) -- (3.9,0.85,0.1);
\draw[myred, thick,dashed] (3,0.85,1) -- (3.9,0.85,1);
\draw[myred, thick,dashed] (3,0.15,1) -- (3.9,0.15,1);

\node at (4.2,-0.1) {$D^k_{\delta^{1-\kappa}}$};
\draw[->] (4.0,-0.1) -- (3.66,-0.1);

\node at (4.2,0.4) {$D^k_{\delta}$};
\draw[->] (4.1,0.4) -- (3.71,0.4);

\draw[black, thick] (current bounding box.south west) rectangle (current bounding box.north east);
\end{tikzpicture}
\caption{Pictorial depiction of $S^i_p$, $A^k_{\delta}$, $D^{k}_{\delta}$ when $\mathcal{X}$ and $\mathcal{Y}$ are two unit cubes where one is a linear translation of other and $\mu,\nu$ are uniform measures on $\mathcal{X},\mathcal{Y}$ respectively.}
\end{figure}

We first prove that
\begin{align}
\pi^\varepsilon(O_{pij}) &\leq \underbrace{\sum_{B_{\delta}(x_k)\subset S^i_p} \int_{A^k_{\delta}\times D^k_{\delta^{1-\kappa}}} \frac{d\pi^\varepsilon}{d(\mu \otimes \nu)}d(\mu \otimes \nu)}_{=:(\mathbf{I})}\nonumber\\  &+ \underbrace{\sum_{B_{\delta}(x_k)\subset S^i_p} \int_{A^k_\delta \times  (D^k_{\delta^{1-\kappa}})^c}\frac{d\pi^\varepsilon}{d(\mu \otimes \nu)}d(\mu \otimes \nu)}_{=:(\mathbf{II})} + o_{\delta}(1) \label{eq:pi(O_{pij})_upper}\\
\pi^\varepsilon(O_{pij}) &\geq \underbrace{\sum_{B_{\delta}(x_k)\subset S^i_p} \int_{A^k_{\delta}\times D^k_\delta} \frac{d\pi^\varepsilon}{d(\mu \otimes \nu)}d(\mu \otimes \nu)}_{=:(\mathbf{III})} - o_{\delta}(1) \label{eq:pi(O_{pij})_lower}
\end{align}
where $(D^k_{\delta^{1-\kappa}})^c$ denotes the complement of the set $D^k_{\delta^{1-\kappa}}$ in $\mathsf{Proj}_{\mathcal{Y}}(O_{pij})$, i.e., $\mathsf{Proj}_{\mathcal{Y}}(O_{pij})\backslash D^k_{\delta^{1-\kappa}}$. The first inequality \eqref{eq:pi(O_{pij})_upper} follows by noticing that 
\begin{align}\label{set:ineq-1}
\pi(\Tilde{Z}_{\delta}) \to 0, \hspace{0.5em}  \Tilde{Z}_{\delta}: =O_{pij} \setminus \bigcup_{B_{\delta}(x_k)\subset S^i_p} \Big( (A_\delta^k \times D_{\delta^{1-\kappa}}^k)  \cup (A_\delta^k \times (D_{\delta^{1-\kappa}}^k)^c)\Big).
\end{align}
To see this, we recall that $\mu_{\llcorner (S^i_p \cap \mathcal{X})}(Z)\to 0$ when $\delta \to 0$ where we define $Z:= (S^i_p\cap \mathcal{X}) \Delta \bigcup^{\ell}_{k=1} B_{\delta}(x_k)$. Through this, one can write 
\begin{align}
    O_{pij} &\subseteq \left(\bigcup_{z\in S_p^i} \mathcal{R}_z^0 \cap \mathsf{Proj}_{\mathcal{X}}(O_{pij})\right) \times \left(\bigcup_{z\in S_p^i} \mathcal{R}_z^0\cap \mathsf{Proj}_{\mathcal{Y}}(O_{pij})\right)\\
    &= \left(\bigcup_{k=1}^{\ell} \bigcup_{z\in B_{\delta}(x_k)} \mathcal{R}_z^0 \cap \mathsf{Proj}_{\mathcal{X}}(O_{pij})\right) \times \left(\bigcup_{k=1}^{\ell} \bigcup_{z\in B_{\delta}(x_k)} \mathcal{R}_z^0 \cap \mathsf{Proj}_{\mathcal{Y}}(O_{pij})\right) \cup \Tilde{Z}\\
    &\subseteq \left(\bigcup_{k=1}^\ell  \bigcup_{m=1}^\ell A_\delta^k \times  D_{\delta^{1-\kappa}}^m\right) \cup \Tilde{Z}\\
    &\subseteq \bigcup_{k=1}^{\ell} (A_\delta^k \times  D_{\delta}^k) \cup \left(A_\delta^k \times (D_{\delta^{1-\kappa}}^k)^c\right)\cup \Tilde{Z},
\end{align}
where 
\begin{align}
\Tilde{Z} &:= \bigcup \left(\bigcup_{z\in  Z} \mathcal{R}_z^0 \cap \mathsf{Proj}_{\mathcal{X}}(O_{pij})\right) \times \left(\bigcup_{z\in S^i_p} \mathcal{R}_z^0\cap \mathsf{Proj}_{\mathcal{Y}}(O_{pij})\right) \\ &\quad \bigcup \left(\bigcup_{z\in  S^i_p} \mathcal{R}_z^0 \cap \mathsf{Proj}_{\mathcal{X}}(O_{pij})\right) \times \left(\bigcup_{z\in Z} \mathcal{R}_z^0\cap \mathsf{Proj}_{\mathcal{Y}}(O_{pij})\right).
\end{align}
This proves \eqref{set:ineq-1}, which in turn implies \eqref{eq:pi(O_{pij})_upper}. 
On the other hand, we have 
\begin{align*}
     \bigcup_{k=1}^\ell A_\delta^k \times  D_{\delta}^k &\subseteq O_{pij} \bigcup \left(\bigcup_{z\in  Z} \mathcal{R}_z^0 \cap \mathsf{Proj}_{\mathcal{X}}(O_{pij})\right) \times \left(\bigcup_{z\in S^i_p} \mathcal{R}_z^0\cap \mathsf{Proj}_{\mathcal{Y}}(O_{pij})\right) \\ &\quad \bigcup \left(\bigcup_{z\in  S^i_p} \mathcal{R}_z^0 \cap \mathsf{Proj}_{\mathcal{X}}(O_{pij})\right) \times \left(\bigcup_{z\in Z} \mathcal{R}_z^0\cap \mathsf{Proj}_{\mathcal{Y}}(O_{pij})\right).
\end{align*}
This proves \eqref{eq:pi(O_{pij})_lower}. We now show that as $\varepsilon \to 0$,
 \begin{align}\label{eq:I-convergence}
    \Big|(\mathbf{I}) &- \delta^{d-1}\sum_{k=1}^{\ell(\delta)} \int_{\mathcal{R} \cap (A^k_{\delta}\times D^k_{\delta^{1-\kappa}})} \Pi_{\mathcal{R}}(x,y)d(\mu_{\llcorner \mathfrak{S}^\mu_{\mathcal{R}}} \otimes \nu_{\llcorner \mathfrak{S}^\nu_{\mathcal{R}}})\Big| \longrightarrow 0 \\
     \Big|(\mathbf{III}) &- \delta^{d-1}\sum_{k=1}^{\ell(\delta)} \int_{\mathcal{R} \cap (A^k_{\delta}\times D^k_{\delta})} \Pi_{\mathcal{R}}(x,y)d(\mu_{\llcorner \mathfrak{S}^\mu_{\mathcal{R}}} \otimes \nu_{\llcorner \mathfrak{S}^\nu_{\mathcal{R}}})\Big| \longrightarrow 0 \label{eq:III-convergence}
\end{align}
and 
\begin{align}\label{eq:II-convergence}
  (\mathbf{II})  \longrightarrow 0.
\end{align}
Before proceeding to the proof of the above claim, let us first show how this implies the desired result of this claim. Recall from \eqref{eq:pi(O_{pij})_upper} and \eqref{eq:pi(O_{pij})_lower} that $\pi^{\varepsilon}(O_{pij})$ is bounded above by $(\mathbf{I})+ (\mathbf{II})$ and lower bounded by $(\mathbf{I})$. Combining these facts with the convergence claimed in the above display proves the result of this lemma. 
\newline
Now we proceed to the proof of \eqref{eq:I-convergence}, \eqref{eq:III-convergence} and \eqref{eq:II-convergence}. Proof of \eqref{eq:I-convergence}, \eqref{eq:III-convergence} are similar. Therefore it suffices to establish \eqref{eq:III-convergence}. We start by proving \eqref{eq:II-convergence}. 
To estimate the second term we make use of Lemma \ref{lem:rate-bound} and our large deviation estimate from Lemma \ref{claim:lower bound} to get that
\begin{align*}
  (\mathbf{II})=\sum_{B_{\delta}\subset S^{i}_p} \pi^{\varepsilon}\left(A_{\delta}\times (D_{\delta^{1-\kappa}})^c \right)& \lesssim \sum_{B_{\delta}\subset S_p^i} \inf_{(a,d)\in A_{\delta}\times D_{\delta^{1-\kappa}}^c} e^{-\frac{I(a,d)}{\varepsilon}}\\
&\lesssim  \delta^{-(d-1)} e^{\frac{-C\delta^{2+2\kappa}}{\varepsilon}},
\end{align*}
where the factor $\delta^{d-1}$ comes from the fact that the intrinsic balls on the level set have volume growth of order $\delta^{d-1}.$ In particular when $\delta = \epsilon^{\frac{1}{2}-\kappa}$, $(\mathbf{II})\to 0$ as $\varepsilon$ tends to zero. Now we proceed to show \eqref{eq:III-convergence}. Recall that 
\begin{align}\label{eq:Treat_I}
 (\mathbf{I}) = \sum_{k=1}^{\ell} \int_{A^k_{\delta}\times D^k_{\delta}} \frac{d\pi^\varepsilon}{d(\mu \otimes \nu)}d(\mu \otimes \nu) .
\end{align}
We need to estimate $\frac{d\pi^\varepsilon}{d(\mu \otimes \nu)}$ on the set $A^k_{\delta}\times D^k_{\delta}$. 
To start with, we write 
\begin{align}\label{eq:pi_epsilon_measure} 
\pi^{\varepsilon}(A^k_\delta\times D^k_{\delta}) = \int_{A_{\delta}\times D_{\delta}} e^{-\frac{I(x',y')}{\varepsilon} + \tilde{f}_{\varepsilon}(x')+ \tilde{g}_{\varepsilon}(y')} d(\mu\otimes \nu). 
\end{align}
where $\tilde{f}_{\varepsilon}(x) = (f_{\varepsilon}(x)+ u(x))/\varepsilon$, $\tilde{g}_{\varepsilon}(y) = (g_{\varepsilon}(y)- u(y))/\varepsilon$ and $u$ is a Kantorovich potential of the underlying optimal transport problem. 

By Lemma~2.2 of \cite{mcann-caff-feldman}, each ray cluster (say $T^0_{pij}$ in this case) admits a local coordinate function $G$ mapping the ray cluster to $\mathbb{R}^{d-1}\times\mathbb{R} $ and its inverse $F$ such that, one has  $F(G(x)) = x$ for all $x\in T^0_{pij}$ and $G(x) = (U(z), u(x)-u(z))$ for any $x$ along any transport ray $\mathcal{R}_z\subset T^0_{pij}$ with endpoint $z$ and $u$ being the Kantorovich potential function. By the change of variable formula, let $p$ and $q$ be the densities of $\mu$ and $\nu$, 
\[\mathrm{d}\mu(x') = p(F(\tilde{x})) J(F(\tilde{x})) d\tilde{x}, \qquad \mathrm{d}\nu(y') = q(F(\tilde{y})) J(F(\tilde{y})) d\tilde{y}\]
where $(\tilde{x},\tilde{y})$ be the local coordinates obtained by the transformation $G$ and $J$ is the Jacobian of the transformation. Furthermore, the construction of the function $F$ in \cite{mcann-caff-feldman} suggests that 
\begin{align}\label{eq:F_decompose} 
F(\mathfrak{x}) = F((\mathfrak{x}_1, \ldots ,\mathfrak{x}_d)) = V(\mathfrak{x}_1,\ldots , \mathfrak{x}_{d-1}) + \mathfrak{x}_d \chi(V(\mathfrak{x}_1.\ldots , \mathfrak{x}_{d-1})) , 
\end{align}
for $\mathfrak{x} \in \mathcal{R}_z$ and some continuous function $V:\mathbb{R}^{d-1} \to \mathbb{R}^d$. Here $\chi$ denotes the direction of transport ray $\mathcal{R}_x$. 
Recall that $A^k_{\delta}$ and $D^k_{\delta}$ are cylinders of radius $\delta$ and around some transport ray $\mathcal{R}_{x_k}$. For the rest of the proof, we use the notations $A_{\delta}$, $D_{\delta}$ and $\mathcal{R}$ in place of $A^k_{\delta}, D^k_{\delta}$ and $\mathcal{R}_{x_k}$ respectively. Since $\delta = O(\varepsilon^{\frac{1}{2}-\kappa})$ and since we assume that $p$ and $q$ are continuous in the interiors of $\mathcal{X}$ and $\mathcal{Y}$, by using \eqref{eq:F_decompose} and continuity of $F$, one can write for any $\tilde{x}=: G(x')\in G(A_{\delta}\times D_{\delta})$ with $\tilde{x} = \mathfrak{x} + \sqrt{\varepsilon}\mathfrak{x}^{\perp}$ and $\tilde{y} = \mathfrak{y} + \varepsilon\mathfrak{y}^{\perp}$  for some $\mathfrak{x}^{\perp}, \mathfrak{y}^{\perp}\in \mathbb{R}^{d-1}$ ($\mathfrak{x}$ and $\mathfrak{y}$ are the projections of $\tilde{x} \in G(A_{\delta})$ and $\tilde{y} \in G(D_{\delta})$  onto the transport ray $R_z$) that  
\begin{align*}
\mathrm{d}\mu(x') & = \varepsilon^{(d-1)/2}p( V(\mathfrak{x}_1.\ldots , \mathfrak{x}_{d-1}) + \mathfrak{x}_d \nu(V(\mathfrak{x}_1.\ldots , \mathfrak{x}_{d-1})))\\ &\times  J(F( V(\mathfrak{x}_1.\ldots , \mathfrak{x}_{d-1}) + \mathfrak{x}_d \chi(V(\mathfrak{x}_1.\ldots , \mathfrak{x}_{d-1})))) (1+O(\delta))d\mathfrak{x}^{\perp}_1 \ldots d\mathfrak{x}^{\perp}_{d-1} d\mathfrak{x}_d
\end{align*}
and,
\begin{align}\label{eqn:local-nu-expression}
\mathrm{d}\nu(y') & = \varepsilon^{(d-1)/2} q( V(\mathfrak{y}_1.\ldots , \mathfrak{y}_{d-1}) + \mathfrak{y}_d \nu(V(\mathfrak{y}_1.\ldots , \mathfrak{y}_{d-1})))\\ &\times  J(F( V(\mathfrak{y}_1.\ldots , \mathfrak{y}_{d-1}) + \mathfrak{y}_d \chi(V(\mathfrak{y}_1.\ldots , \mathfrak{y}_{d-1})))) (1+O(\delta))d\mathfrak{y}^{\perp}_1 \ldots d\mathfrak{y}^{\perp}_{d-1} d\mathfrak{y}_d.
\end{align}
By Taylor's expansion of $I(x',y')$, we have 
\begin{align}\label{eqn:I-taylor}
I(x',y') = \|F(\tilde{x}) - F(\tilde{y})\| - \|F(\mathfrak{x}) - F(\mathfrak{y})\| = 
\varepsilon \frac{\|(\overline{\tfrac{\partial F(\mathfrak{x})}{\partial \mathfrak{x}}})\mathfrak{x}^{\perp} - (\overline{\tfrac{\partial F(\mathfrak{y})}{\partial \mathfrak{y}}})\mathfrak{y}^{\perp}\|^2}{2\|F(\mathfrak{x})-F(\mathfrak{y})\|} + O(\varepsilon\delta)    
\end{align}
where the last equality follows by noting that $\langle F(\mathfrak{x}) - F(\mathfrak{y}), (\overline{\tfrac{\partial F(\mathfrak{x})}{\partial \mathfrak{x}}})\mathfrak{x}^{\perp} - (\overline{\tfrac{\partial F(\mathfrak{y})}{\partial \mathfrak{y}}})\mathfrak{y}^{\perp} \rangle = O(\varepsilon\delta)$. Here $\overline{\tfrac{\partial F(\mathfrak{x})}{\partial \mathfrak{x}}}$ and $\overline{\tfrac{\partial F(\mathfrak{y})}{\partial \mathfrak{y}}}$ denote the principle submatrices $(d-1)\times (d-1)$ of $\tfrac{\partial F(\mathfrak{x})}{\partial \mathfrak{x}}$ and $\tfrac{\partial F(\mathfrak{y})}{\partial \mathfrak{y}}$, respectively.
Furthermore, from Theorem~\ref{lem:Deri_Conv}(b), we have 
\begin{align}\label{eqn:f_eps and g_eps expansion}
\tilde{f}_{\varepsilon}(x') &= \tilde{f}_{\varepsilon}(F(\mathfrak{x})) + \frac{1}{\sqrt{\varepsilon}}\langle v_1, (\overline{\tfrac{\partial F(\mathfrak{x})}{\partial \mathfrak{x}}})\mathfrak{x}^{\perp}\rangle + O(\sqrt{\varepsilon}) \\ 
\tilde{g}_{\varepsilon}(y') &= \tilde{g}_{\varepsilon}(F(\mathfrak{y})) - \frac{1}{\sqrt{\varepsilon}}\langle v_1, (\overline{\tfrac{\partial F(\mathfrak{y})}{\partial \mathfrak{y}}})\mathfrak{y}^{\perp}\rangle + O(\sqrt{\varepsilon}),
\end{align}
for some constant vector $v_1 = c\cdot (1,\ldots, 1) \in \mathbb{R}^d$.
Plugging  these into the right hand side of \eqref{eq:pi_epsilon_measure} and substituting $\bar{\mathfrak{x}}^\perp = (\overline{\tfrac{\partial F(\mathfrak{x})}{\partial \mathfrak{x}}})\mathfrak{x}^{\perp}$, $\bar{\mathfrak{y}}^{\perp} = (\overline{\tfrac{\partial F(\mathfrak{y})}{\partial \mathfrak{y}}})\mathfrak{y}^{\perp}$ (and also recognizing $F(\mathfrak{x})= x$ and $F(\mathfrak{y})= y$) yields 
\begin{align}\label{eqn:integral-on-rays}
\pi^{\varepsilon}(A_\delta\times D_{\delta}) = &\int_{(A_{\delta}\times D_{\delta})\cap \mathcal{R}}  \int^{\delta\varepsilon^{-1/2}}_{-\delta\varepsilon^{-1/2}}\int^{\delta\varepsilon^{-1/2}}_{-\delta\varepsilon^{-1/2}} \int^{\delta\varepsilon^{-1/2}}_{-\delta\varepsilon^{-1/2}}\int^{\delta\varepsilon^{-1/2}}_{-\delta\varepsilon^{-1/2}} \\ & \times \varepsilon^{(d-1)}   e^{- \frac{\|\bar{\mathfrak{x}}^{\perp} - \bar{\mathfrak{y}}^{\perp}\|^2}{2\|x-y\|} + \langle v_1, (\bar{\mathfrak{x}}^\perp - \bar{\mathfrak{y}}^\perp)\rangle + O(\sqrt{\varepsilon})} \\ & \times p(x)q(y) e^{\tilde{f}_{\varepsilon}(x)+ \tilde{g}_{\varepsilon}(y) }  d\bar{\mathfrak{x}}^\perp_1 \ldots d\bar{\mathfrak{x}}^\perp_{d-1} d\mathfrak{x}_d d\bar{\mathfrak{y}}^\perp_1 \ldots d\bar{\mathfrak{y}}^\perp_{d-1} d\mathfrak{y}_d.
\end{align}
 Notice that $x$ and $y$ are parametrized as $x = V(z_1,\ldots , z_{d-1}) + \mathfrak{x}_\mathrm{d}\chi(z_1,\ldots , z_{d-1})$ and $y = V(z_1,\ldots , z_{d-1}) + \mathfrak{y}_\mathrm{d}\chi(z_1,\ldots , z_{d-1}, z_d)$. Thus
\begin{align}
\|x-y\| = \|\chi(z_1,\ldots , z_{d-1}, z_d)\| |\mathfrak{x}_d - \mathfrak{y}_d| = |\mathfrak{x}_d - \mathfrak{y}_d|.    
\end{align}
Completing the integral with respect to $x^{\perp}$ and $y^{\perp}$, we obtain that
\begin{align}
\pi^{\varepsilon}(A_\delta\times D_{\delta}) &\sim \delta^{(d-1)} \varepsilon^{(d-1)/2}\int_{(A_{\delta}\times D_{\delta})\cap \mathcal{R}} (2\pi \|\mathfrak{x}_d-\mathfrak{y}_d\|)^{\frac{d-1}{2}} e^{c^2(d-1)\|\mathfrak{x}_d-\mathfrak{y}_d\|} \\ & \times e^{\tilde{f}_{\varepsilon}(x)+ \tilde{g}_{\varepsilon}(y) } p(x) q(y) d\mathfrak{x}_d d\mathfrak{y}_d
\end{align}
where $\mathfrak{x}_d$ and $\mathfrak{y}_d$ vary in the set $A_{\delta}\cap R_z$ and $D_{\delta}\cap R_z$ respectively. 
 Theorem~\ref{lem:Deri_Conv}(a) further implies that $(\tilde{f}_{\varepsilon}(x)+ \tilde{g}_{\varepsilon}(y) + \frac{d-1}{2}\log \varepsilon)$ converges uniformly on any compact sets of $\mathcal{R}$ to $\mathfrak{f}(x)+ \mathfrak{g}(y)$ as $\varepsilon \to 0$ where $\mathfrak{f}$ and $\mathfrak{g}$ are uniquely determined (up to some constants) by the relation \eqref{eq:mathfrak_f} and \eqref{eq:mathfrak_g}. Applying this convergence result to the above display shows that the integrand converges pointwise to $\Pi_{\mathcal{R}} (x,y)$. Plugging this into the right-hand side of \eqref{eq:Treat_I} yields \eqref{eq:II-convergence}. This completes the proof. 

\section{Limit Theory of EOT Potentials}\label{sec:eot}
In this section, we establish the expansion of the entropic optimal transport potentials in the vicinity of a transport ray. Earlier works, such as \cite{Nutz-Wiesel}, have demonstrated that the EOT potentials converge to the Kantorovich potential as the entropic regularization parameter $\varepsilon \to 0$. Building on this foundation, Theorem~\ref{lem:Deri_Conv} extends these results by showing that the scaled difference between the EOT potential and the Kantorovich potential, when divided by $\varepsilon$, has a non-trivial limit as $\varepsilon \to 0$. In addition, Theorem~\ref{lem:Deri_Conv} establishes the convergence of the derivative (scaled by $\varepsilon$) of the EOT potential in any direction orthogonal to a transport ray. These findings provide a more refined understanding of the behavior of the EOT potential near transport rays, further enriching the theoretical framework of entropic optimal transport under vanishing regularization. 
\begin{theorem}\label{lem:Deri_Conv}
Let $(f_{\varepsilon}, g_{\varepsilon})$ be the entropic potentials and $(-u,u)$ be Monge-Kantorovich potentials as in Section \ref{sec:l1-ot}. Denote
\begin{align*}
\tilde{f}_{\varepsilon}(x) = \frac{f_{\varepsilon}(x)+ u(x)}{\varepsilon}, \quad \tilde{g}_{\varepsilon}(x) =  \frac{g_{\varepsilon}(x)- u(x)}{\varepsilon}.
\end{align*}
  For any transport ray $\mathcal{R}$, let $(v_1,\ldots, v_{d-1})$ be the set of $d-1$ orthonormal vectors which are also orthogonal to the tangential direction on $\mathcal{R}$.  For any point $x \in \mathcal{R}^{\circ}\cap \mathcal{X}$ and $y\in \mathcal{R}^{\circ}\cap\mathcal{Y}$, we define 
\begin{align}
    A^g_\varepsilon(y) &:= \frac{1}{\sqrt{\varepsilon}}(\langle \nabla g_{\varepsilon}(y), v_1\rangle, \ldots  \langle \nabla g_{\varepsilon}(y), v_{d-1}\rangle), \\
    A^f_\varepsilon(x) &:= \frac{1}{\sqrt{\varepsilon}} (\langle \nabla f_{\varepsilon}(x), v_1\rangle, \ldots  \langle \nabla f_{\varepsilon}(x), v_{d-1}\rangle).
\end{align}
For almost all transport ray $\mathcal{R}$, we have:
\begin{enumerate}
\item[(a)]  $(\tilde{f}_{\varepsilon} + \frac{(d-1)}{2}\log \varepsilon, \tilde{g}_{\varepsilon})$ converges uniformly on any compact subset of $(\mathcal{R}^\circ\cap \mathcal{X})\times (\mathcal{R}^\circ\cap \mathcal{Y})$.

\item[(b)] As $\varepsilon \to 0$, $\lim_{\varepsilon \to 0 }A^g_{\varepsilon}(y) = c(1,\ldots, 1)$ and $\lim_{\varepsilon \to 0} A^f_\varepsilon(x) = - c(1,\ldots, 1)$ for some $c\in \mathbb{R}$ which is same for all $x\in \mathcal{R}^{\circ}$.  
\end{enumerate}
\end{theorem}
 We establish the above theorem by relying on the following three key lemmas. First, we state these lemmas and then proceed with the proof of Theorem~\ref{lem:Deri_Conv}. The proofs of the lemmas will be presented subsequently.

\begin{lemma}\label{lem:const-derivative}
Let $\mathcal{R}$ be a transport ray with a direction vector $\nu\in \R^{d}$ and let $(x,y)\in \mathcal{X}\times \mathcal{Y}$ be interior points on $\mathcal{R}$ such that $u(x)-u(y)=\|x-y\|$. Let $\{e_1,\cdots,e_{n-1},e_{n}\}$ be an orthonormal frame with $e_n$ is the direction of the transport ray $\mathcal{R}$. Then as $\varepsilon\to 0$ we have
\begin{align}
\partial_{i}f_{\varepsilon}(x) &\to  c,\quad \partial_{j}g_{\varepsilon}(y) \to -c
\end{align}
for indices $1\leq i,j\leq n-1$ and $c$ is some constant independent of $i,j.$
\end{lemma}

\begin{lemma}\label{lem:constant-on-strips}
 There exists continuous bounded functions $\Theta^i_\varepsilon:\R^n\times \R\to \R$, $i=1,2$ such that the following identities hold
\begin{align}
\Tilde{f}_\varepsilon(x) &= -\frac{(d-1)}{2}\log\varepsilon  -\log\left(\int_{\mathcal{R}_x} \Theta^1_\varepsilon(x,y_x)e^{\Tilde{g}_\varepsilon(y_x)} \mathrm{d}\nu_{\llcorner \mathcal{R}_x}(y_x)\right)  + o_{\varepsilon}(1) \label{eqn:tilde-f}\\
\Tilde{g}_\varepsilon(y) &= -\frac{(d-1)}{2}\log\varepsilon  - \log\left(\int_{\mathcal{R}_y} \Theta^2_\varepsilon(x_y,y)e^{\Tilde{f}_\varepsilon(x_y)} \mathrm{d}\mu_{\llcorner \mathcal{R}_y}(x_y)\right) +  o_{\varepsilon}(1) \quad \quad
\label{eqn:tilde-g}
\end{align}
where $\Tilde{f}_\varepsilon(x)=\frac{f_\varepsilon(x)+u(x)}{\varepsilon}$, $\Tilde{g}_\varepsilon(y)=\frac{g_\varepsilon(y)-u(y)}{\varepsilon}$ for a.e. $x\in \mathcal{X}$ and $y\in \mathcal{Y}$. 
\end{lemma}

\begin{lemma}\label{lem:entropic-cost-expansion}
Let $v_\varepsilon$ be defined as in \eqref{defn:dual-v_eps}. Then for $\varepsilon>0$ small enough we have 
\begin{align}
    v_{\varepsilon}-v_{0} \geq \frac{(d-1)}{2}\varepsilon \log(1/\varepsilon)+O(\varepsilon).
\end{align}
\end{lemma}

\begin{proof}[Proof of Theorem~\ref{lem:Deri_Conv}]
 We first show that $A^g_{\varepsilon}, A^f_{\varepsilon}$ are bounded. Using this fact, we prove later the convergence of  $(\tilde{f}_{\varepsilon} + \frac{(d-1)}{2}\log \varepsilon, \tilde{g}_{\varepsilon})$, $A^g_{\varepsilon}$, and $A^f_{\varepsilon}$. From Lemma \ref{lem:constant-on-strips},  for $x\in \mathcal{R}^\circ \cap \mathcal{X}$ and $y\in \mathcal{R}^\circ \cap \mathcal{Y}$,
\begin{align}\label{eq:fg_combined}
  &\Tilde{f}_{\varepsilon}(x) +\Tilde{g}_{\varepsilon}(y) +(d-1)\log(\varepsilon)  + \log\left(\int_{\mathcal{R}^\circ\cap \mathcal{Y}} \Theta^1_\varepsilon(x,y')e^{\Tilde{g}_\varepsilon(y)} q(y') \ \mathrm{d}y'\right) \\
  &\quad + \log\left(\int_{\mathcal{R}^{\circ}\cap \mathcal{X}} \Theta^2_\varepsilon(x',y)e^{\Tilde{f}_{\varepsilon}(x')} p(x') \ \mathrm{d}x'\right) + o_{\varepsilon}(1)= 0 \label{eq:ray_equation} 
\end{align}
where $\Tilde{f}_{\varepsilon}=\frac{f_{\varepsilon}+u}{\varepsilon}$ and $\Tilde{g}_\varepsilon=\frac{g_{\varepsilon}-u}{\varepsilon}$. From Lemma \ref{lem:entropic-cost-expansion}, we know that 
\begin{align}\label{eq:f_g}
   \int \Tilde{f}_{\varepsilon}(x) \ \mathrm{d}\mu(x) +\int \Tilde{g}_{\varepsilon}(y) \ \mathrm{d}\nu(y)  \geq -\frac{(d-1)}{2} \log \varepsilon + O(1).
\end{align}
Thus we get
\begin{align}
    \frac{(d-1)}{2}\log \varepsilon & + \int \log\left(\int_{\mathcal{R}^\circ_x\cap \mathcal{Y}} \Theta^1_\varepsilon(x,y')e^{\Tilde{g}_{\varepsilon}} q(y')dy'\right)\mathrm{d}\mu(x) \\ & + \int \log\left(\int_{\mathcal{R}^\circ_y\cap \mathcal{X}} \Theta^2_\varepsilon(x',y)e^{\Tilde{f}_{\varepsilon}} p(x')dx'\right) \mathrm{d}\nu(y) \leq O(1),
\end{align}
where $\mathcal{R}^\circ_x$ denotes the interior of the transport containing $x$. 
Now writing the integrals in the logarithms as expectations with respect to the measures $\Theta^i_\varepsilon$ for $i=1,2$ we get
\begin{align}
    &\frac{(d-1)}{2}\log \varepsilon + \int\log \mathbb{E}_{\Theta^1_\varepsilon}\big[e^{\Tilde{g}_{\varepsilon}(Y)}\big] \mathrm{d}\mu(x) +\int \Big(\log \int_{\mathcal{R}^\circ_x\cap \mathcal{Y} } \Theta^1_\varepsilon(x,y')q(y') \mathrm{d}y'\Big) \mathrm{d}\mu(x)  \\
    &\quad + \int\log \mathbb{E}_{\Theta^2_\varepsilon}\big[e^{\Tilde{f}_{\varepsilon}(X)}\big] \mathrm{d}\nu(y) +\int \log \Big(\int_{\mathcal{R}^\circ_y\cap \mathcal{X} } \Theta^2_\varepsilon(x',y)p(x') \mathrm{d} x'\Big) \mathrm{d}\nu(y)\leq O(1),
\end{align}
where $X$ and $Y$ are two random variables supported on $\mathcal{R}^{\circ}_x\cap \mathcal{X}$ and $\mathcal{R}^{\circ}_y \cap \mathcal{Y}$ with densities proportional to $\Theta^2_{\varepsilon}(\cdot, y)$ and $\Theta^2_{\varepsilon}(x, \cdot)$ with respect to $\mu_{{\llcorner \mathcal{R}^{\circ}}_x}$ and $\nu_{{\llcorner \mathcal{R}^{\circ}}_y}$ respectively. Applying Jensen's inequality, we thus have
\begin{align}
    &\frac{(n-1)}{2}\log \varepsilon +  \int \mathbb{E}_{\Theta^1_\varepsilon}\left[\Tilde{g}_{\varepsilon}(Y)\right] \mathrm{d}\mu(x) +\int\Big(\log \int_{\mathcal{R}^\circ_x\cap \mathcal{Y} } \Theta^1_\varepsilon(x,y') q(y')dy'\Big)\mathrm{d}\mu(x)   \\ &+ \int\mathbb{E}_{\Theta^2_\varepsilon}\left[\Tilde{f}_{\varepsilon}(X)\right]\mathrm{d}\nu(x) +\int \Big(\log \int_{\mathcal{R}^\circ_y\cap \mathcal{X} } \Theta^2_\varepsilon(x',y)p(x') dx'\Big)\mathrm{d}\nu(y)\leq O(1).
\end{align}
Observe that
$$ \int \mathbb{E}_{\Theta^1_\varepsilon}\left[\Tilde{g}_{\varepsilon}(Y)\right] \mathrm{d}\mu(x) + \int\mathbb{E}_{\Theta^2_\varepsilon}\left[\Tilde{f}_{\varepsilon}(X)\right]\mathrm{d}\nu(x) = \int(\Tilde{f}_{\varepsilon}(x) + \Tilde{g}_{\varepsilon}(y)) \ \mathrm{d}\hat{\pi}^{\mathrm{opt}}(x,y)$$
for some optimal transport plan $\hat{\pi}^{\mathrm{opt}}$. 
Due to the inequality in \eqref{eq:f_g}, we get $\int\mathbb{E}_{\Theta^1_\varepsilon}\left[\Tilde{g}_{\varepsilon}(Y)\right] \mathrm{d}\mu(x) + \int \mathbb{E}_{\Theta^2_\varepsilon}\big[\Tilde{f}_{\varepsilon}(X)\big] \mathrm{d}\nu(y) + \frac{1}{2}(d-1)\log \varepsilon\geq O(1)$. Using this inequality in the above display yields
\begin{align}\label{eq:Theta_1&2_bound}
   \int \log \Big(\mathbb{E}_{\nu_{\llcorner \mathcal{R}^\circ_x}}\big[\Theta^1_\varepsilon(x,\cdot)\big] \Big) \mathrm{d}\mu(x)  + \int \log\Big( \mathbb{E}_{\mu_{\llcorner \mathcal{R}^\circ_y}}\big[\Theta^2_\varepsilon(\cdot,y)\big]\Big)\mathrm{d}\nu(y)\leq O_{\varepsilon}(1).\quad 
\end{align}
We claim that the above implies $A^{g}_{\varepsilon}$ and $A^{f}_{\varepsilon}$ are almost everywhere bounded with respect to $\nu_{\llcorner \mathcal{R}^{\circ}}$ and $\mu_{\llcorner \mathcal{R}^{\circ}}$ respectively as $\varepsilon \to 0$ for almost all transport ray $\mathcal{R}$. We show the implication for $A^{g}_{\varepsilon}$. Other case follows similarly. Recall from \eqref{eq:Theta_1} of Lemma~\ref{lem:constant-on-strips} that $$\Theta^1_\varepsilon (x,y') = (2\pi \|x-y'\|)^{\frac{d-1}{2}} e^{\|A^g_{\varepsilon}(y')\|^2\|x-y'\|/2}e^{\inner{A^g_{\varepsilon}(y')}{x}}.$$ Plugging this expression into \eqref{eq:Theta_1&2_bound} indicates the almost everywhere boundedness of $A^{g}_{\varepsilon}$ with respect to $\nu_{\llcorner \mathcal{R}^{\circ}}$ as $\varepsilon \to 0$ for almost all transport ray $\mathcal{R}$.  

\noindent \textbf{Proof of (a)}: Recall that $I(x,y)= c(x,y) + u(x) -u(y)$. For $\delta>0$ small, from \eqref{eqn:localized-integral}, the following expansion holds when $\varepsilon \to 0$,
\begin{align}\label{eq:Duality}
\Tilde{f}_\varepsilon(x)+\log \int_{D_\delta} e^{\frac{-I(x,y)+g_{\varepsilon}(y)-u(y)}{\varepsilon}} \mathrm{d}\nu(y) + o_\varepsilon(1) = 0.
\end{align}
Recall the pair of functions $(F,G)$ from Lemma~\ref{lem:lip-change-of-variable}. Furthermore, expanding as in Lemma~\ref{lemma:Main}, we have
\begin{align}
\tilde{f}_\varepsilon(x) + \frac{d-1}{2} \log \varepsilon &=  -\log\left( \int_{G(D_\delta)} e^{-\frac{\|(\overline{\partial_{\mathfrak{x}} F(\mathfrak{x})})\mathfrak{x}^\perp - (\overline{\partial_{\mathfrak{y}}F(\mathfrak{y})}) \mathfrak{y}^\perp\|^2}{2\|F(\mathfrak{x})-F(\mathfrak{y})\|}+O(\delta)} e^{\Tilde{g}_{\varepsilon}(F(\mathfrak{y}+\sqrt{\varepsilon}\mathfrak{y}^\perp))} \right.\\
&\quad \quad \left.\times q(F(\mathfrak{y})) J(F(\mathfrak{y}))(1+O(\delta)) \ \mathrm{d}\mathfrak{y}_1^\perp \cdots \mathrm{d}\mathfrak{y}_{d-1}^\perp\mathrm{d}\mathfrak{y}_d\right) + o_\varepsilon(1)
\end{align}
and similarly,
\begin{align}
\Tilde{g}_\varepsilon(y)  &= -\log \varepsilon^{(d-1/2)}\left( \int_{G(A_\delta)} e^{-\frac{\|(\overline{\partial_{\mathfrak{x}} F(\mathfrak{x})})\mathfrak{x}^\perp - (\overline{\partial_{\mathfrak{y}}F(\mathfrak{y})}) \mathfrak{y}^\perp\|^2}{2\|F(\mathfrak{x})-F(\mathfrak{y})\|}+O(\delta)} e^{\Tilde{f}_{\varepsilon}(F(\mathfrak{x}+\sqrt{\varepsilon}\mathfrak{x}^\perp))} \right.\\
&\quad \quad \left.\times p(F(\mathfrak{x})) J(F(\mathfrak{x}))(1+O(\delta)) \ \mathrm{d}\mathfrak{x}_1^\perp \cdots \mathrm{d}\mathfrak{x}_{d-1}^\perp\mathrm{d}\mathfrak{x}_d\right) + o_\varepsilon(1)\\
&= -\log \left( \int_{G(A_\delta)} e^{-\frac{\|(\overline{\partial_{\mathfrak{x}} F(\mathfrak{x})})\mathfrak{x}^\perp - (\overline{\partial_{\mathfrak{y}}F(\mathfrak{y})}) \mathfrak{y}^\perp\|^2}{2\|F(\mathfrak{x})-F(\mathfrak{y})\|}+O(\delta)} e^{\Tilde{f}_{\varepsilon}(F(\mathfrak{x}+\sqrt{\varepsilon}\mathfrak{x}^\perp))+\frac{(d-1)}{2}\log \varepsilon} \right.\\
&\quad \quad \left.\times p(F(\mathfrak{x})) J(F(\mathfrak{x}))(1+O(\delta)) \ \mathrm{d}\mathfrak{x}_1^\perp \cdots \mathrm{d}\mathfrak{x}_{d-1}^\perp\mathrm{d}\mathfrak{x}_d\right) + o_\varepsilon(1).
\end{align}
The above two equations show that $\Tilde{f}_\varepsilon+\frac{(d-1)}{2}\log\varepsilon$ and $\Tilde{g}_\varepsilon$ are two entropic optimal transport potential for the entropic optimal transport problem between $\tilde{\mu}_{\varepsilon}=G_{\#}\mu_{\llcorner D_\delta}$ and $\tilde{\nu}_{\varepsilon}=G_{\#}\nu_{\llcorner D_\delta}$ in the space $G(A_\delta) \times G(D_\delta)$ with respect to the cost function $c_{\varepsilon} (x,y) = \frac{\|(\overline{\partial_{\mathfrak{x}} F})\mathfrak{x}^\perp-(\overline{\partial_{\mathfrak{y}}F})\mathfrak{y}^\perp\|^2}{2\|F(\mathfrak{x})-F(\mathfrak{y})\|}+ O(\varepsilon)$. Since $A^{f}_{\varepsilon}$ and $A^{g}_{\varepsilon}$ are uniformly bounded, $\Tilde{f}_{\varepsilon}(F(\mathfrak{x}+\sqrt{\varepsilon}\mathfrak{x}^\perp))$ and $\Tilde{g}_{\varepsilon}(F(\mathfrak{x}+\sqrt{\varepsilon}\mathfrak{x}^\perp))$ grow linearly in $\mathfrak{x}^\perp$ and $\mathfrak{y}^\perp$. As a result, the integrals in the above two displays are uniformly bounded on compact sets as $\varepsilon \to 0$. Thus $\Tilde{f}_\varepsilon+\frac{(d-1)}{2}\log\varepsilon$ and $\Tilde{g}_\varepsilon$ are tight on any compact subset of $\mathcal{R}^{\circ}$. Notice that  $\tilde{\mu}_{\varepsilon}$ weakly converges to $\mu_{\llcorner \mathcal{R}^{\circ}}\otimes \mathrm{Leb}(\mathbb{R}^{d-1})$ and $\tilde{\nu}_{\epsilon}$ weakly converges to $\nu_{\llcorner \mathcal{R}^{\circ}}\otimes \mathrm{Leb}(\mathbb{R}^{d-1})$ and furthermore, $c_{\varepsilon}(x,y)$ converges to $\frac{\|(\overline{\partial_{\mathfrak{x}} F})\mathfrak{x}^\perp-(\overline{\partial_{\mathfrak{y}}F})\mathfrak{y}^\perp\|^2}{2\|F(\mathfrak{x})-F(\mathfrak{y})\|}$ as $\varepsilon \to 0$. By using the tightness as argued before and  Theorem 2.5 in~\cite{stability}, we conclude that $\Tilde{f}_\varepsilon(x)+\frac{(d-1)}{2}\log\varepsilon$ and $\Tilde{g}_\varepsilon(y)$ jointly converge to the potential function $(\tilde{f}(x), \tilde{g}(y))$ for any $x,y$ connected by the transport ray $\mathcal{R}$ where $(\tilde{f}(x), \tilde{g}(y))$ are potential functions.

\noindent\textbf{Proof of (b):}  From part (a), we know $\Tilde{f}_{\varepsilon}+\frac{1}{2} (d-1)\log \varepsilon$ and $\Tilde{g}_{\varepsilon}$ converges uniformly on compacts of almost all transport ray $\mathcal{R}$ as $\varepsilon \to 0$. Combining this convergence along with the boundedness of the sequence $A^{g}_{\varepsilon}$ and $A^{f}_{\varepsilon}$ and \eqref{eq:fg_combined} yields the convergence of $A^{g}_{\varepsilon}$ and $A^{f}_{\varepsilon}$. Now convergence of $A^{g}_{\varepsilon}$ and $A^{f}_{\varepsilon}$ to $c(1,\ldots, 1)$ and $-c(1,\ldots, 1)$ for some constant $c\in \mathbb{R}$ follows from Lemma~\ref{lem:const-derivative}.

\end{proof}

\subsection{Proof of Lemma~\ref{lem:const-derivative}}
Let $(f_\varepsilon,g_\varepsilon)$ and $(f_K,g_K)$ denote extremizers for the dual problems \eqref{defn:dual-v_eps} and \eqref{defn:dual-v0}. Then from Proposition 4 in \cite{mcann-caff-feldman} there exists $u\in \operatorname{Lip}_1(\R^d)$ such that
\begin{align}
    u(x) = -f_K(x),\quad u(y) = g_K(y), \quad \forall (x,y)\in \mathcal{X}\times \mathcal{Y}.
\end{align}
Using the relation
\begin{align}
f_{\varepsilon}(x) &=-\varepsilon \log \int_{\mathcal{Y}} e^{\frac{g_{\varepsilon}(y)-\|x-y\|}{\varepsilon}} \ \mathrm{d}\nu(y),\quad  \mu \text {-a.s. }
\end{align}
we get
\begin{align}
\frac{f_{\varepsilon}(x)+u(x)}{\varepsilon}+\log \int_{\mathcal{Y}} e^{\frac{-I(x,y)+g_{\varepsilon}(y)-u(y)}{\varepsilon}} \ \mathrm{d}\nu(y) = 0 \label{eqn:f_eps}
\end{align}
where $I(x,y) = \|x-y\|+u(x)-u(y).$ For any given $x$, we denote $y_x $ to be any $y\in \mathcal{R}_x\cap \mathcal{Y}$ such that $I(x,y)\equiv 0$. Furthermore, for points outside a $\delta$-neighbourhood of the transport ray, there exists $c_0=c_0(\delta)>0$ such that $I(x,y)\geq c_0 > 0.$ Therefore, if we assume that $x\in S^i_p$ for some $i\in \mathbb{N},p\in \mathbb{Q}$, it suffices to estimate the integral on the domain $D_\delta = \bigcup_{z\in B_\delta(x)\cap S^i_p} \mathcal{R}_z \cap \mathcal{Y}$,
\begin{align}\label{eqn:localized-integral}
\frac{f_{\varepsilon}(x)+u(x)}{\varepsilon} = -\log \int_{D_\delta} e^{\frac{-I(x,y)+g_{\varepsilon}(y)-u(y)}{\varepsilon}} \ \mathrm{d}\nu(y) + O(\varepsilon^{-1}c_0(\delta)),
\end{align}
where we also used the fact that the potentials $g_\varepsilon$ and $u$ are bounded on compact sets. Consider variations of $x$ along the level set $S^i_p,$ i.e. $x_\delta = \operatorname{Proj}_{S^i_p}(x+\delta \chi^\perp)$ where $\chi^\perp$ is an unit vector orthogonal to the direction of transport ray $\mathcal{R}_x$. We estimate the derivative in the normal direction of the equation \eqref{eqn:localized-integral} as follows
\begin{align}\label{eqn:df_eps}
&\frac{f_{\varepsilon}(x_\delta)+u(x_\delta) - (f_{\varepsilon}(x)+u(x))}{ \varepsilon} \\
&= \frac{\delta}{\varepsilon} \inner{\nabla f_\varepsilon(x)}{\chi^\perp} +  O(\varepsilon^{-1}\delta)
\end{align}
where we used the fact that $u(x_\delta)-u(x)=O(\delta)$ since $\inner{\chi^\perp}{\nabla u(x)}=0.$ From the integral term in \eqref{eqn:localized-integral}, we have
\begin{align}\label{eqn:expansion}
&\int_{D_\delta} \left(e^{\frac{f(x_\delta,y)}{\varepsilon}} -e^{\frac{f(x,y)}{\varepsilon}}\right) \ \mathrm{d}\nu(y) \\
&\int_{D_\delta} e^{\frac{f(x,y)}{\varepsilon}}\left(e^{\frac{f(x_\delta,y)-f(x,y)}{\varepsilon}} -1\right) \ \mathrm{d}\nu(y) \\
&=  \int_{D_\delta} \left(\frac{f(x_\delta,y)-f(x,y)}{\varepsilon} + O(\varepsilon^{-1}\delta)\right)e^{\frac{f(x,y)}{\varepsilon}}\ \mathrm{d}\nu(y),
\end{align}
where $f$ is defined as follows
\begin{align}\label{defn:f}
f(x,y) &:= -I(x,y)+g_{\varepsilon}(y)-u(y). 
\end{align}
Recall that Lemma \ref{lem:lip-change-of-variable} gives local coordinate function $G$ mapping a ray cluster to $\R^{d-1}\times \R$ along with its inverse map $F.$ Thus, we can express $x\in \mathcal{X}$ and $y\in D_\delta$ as
\begin{align}
\tilde{x}=G(x),\quad  \tilde{x}=\mathfrak{x},\quad \tilde{y}=G(y),\quad \tilde{y}=\mathfrak{y}+\mathfrak{y}^\perp
\end{align}
where $\mathfrak{x}$ and $\mathfrak{y}$ are orthogonal projections of the points $\tilde{x},\tilde{y}$ onto the transport $\mathcal{R}_x$. Expanding $I$ and $g_\varepsilon$ as in \eqref{eqn:I-taylor} and \eqref{eqn:f_eps and g_eps expansion} we get
\begin{align}
f(x,y) = -\frac{\|(\overline{\tfrac{\partial F(\mathfrak{y})}{\partial \mathfrak{y}}})\mathfrak{y}^{\perp}\|^2}{2\|F(\mathfrak{x})-F(\mathfrak{y})\|} + g_\varepsilon(F(\mathfrak{{y}}))+\inner{\nabla g_\varepsilon(F(\mathfrak{y}))}{(\overline{\tfrac{\partial F(\mathfrak{y})}{\partial \mathfrak{y}}})\mathfrak{y}^\perp}  -u(F(\mathfrak{y})) + O(\delta).
\end{align}
Similarly, we can write
\begin{align}
\widetilde{x_\delta}=G(x_\delta), \quad \widetilde{x_\delta}= \mathfrak{\bar{x}} ,\quad \tilde{y} =\mathfrak{\mathfrak{\bar{y}}}+\mathfrak{\bar{y}}^\perp
\end{align}
where $\mathfrak{\bar{x}},\mathfrak{\bar{y}}$ are orthogonal projections of the points $\widetilde{x_\delta}$ and $\tilde{y}$ on the transport ray $\mathcal{R}_{x_\delta}$. Thus, we get 
\begin{align}
f(x_\delta,y) &   = -\frac{\| (\overline{\tfrac{\partial F(\mathfrak{\bar{y}})}{\partial \mathfrak{{y}}}}){\mathfrak{\bar{y}}}^{\perp}\|^2}{2\|F(\mathfrak{\bar{x}})-F(\mathfrak{\bar{y}})\|}+ g_\varepsilon(F(\mathfrak{\bar{y}}))+\inner{\nabla g_\varepsilon(F(\mathfrak{\bar{y}}))}{(\overline{\tfrac{\partial F(\mathfrak{\bar{y}})}{\partial \mathfrak{y}}})\mathfrak{\bar{y}}^\perp} -u(F(\mathfrak{\bar{y}}))+ O(\delta).
\end{align}
By Lemma \ref{lem:lip-control} the distance between the transport rays emanating from $x$ and $x_\delta$ is of order $O(\delta)$, i.e. $\|\mathfrak{y}-\mathfrak{\bar{y}}\|+\|\mathfrak{\bar{y}}^\perp-\mathfrak{y}^\perp\|=O(\delta)$. Furthermore, since $\|x_\delta-x\|=O(\delta)$ and $G$ is a Lipschitz map we have $\|\mathfrak{x}-\mathfrak{\bar{x}}\|=O(\delta).$ Thus, if we further decompose $\mathfrak{\bar{y}}^\perp=\mathfrak{z}+\mathfrak{z}^\perp$, where $\mathfrak{z}^\perp$ is orthogonal to $\mathcal{R}_x$, and omit higher order terms, we get
\begin{align}
    f(x_\delta,y)-f(x,y) = \delta \inner{\nabla g_\varepsilon(F(\mathfrak{y}))}{\tilde{\chi}^\perp} + o(\delta)
\end{align}
where $\tilde{\chi}^\perp =\frac{\mathfrak{\bar{z}}'^\perp-\mathfrak{{y}}'^\perp}{\|\mathfrak{\bar{z}}'^\perp-\mathfrak{{y}}'^\perp\|}$, $\mathfrak{\bar{z}}'^\perp=(\overline{\tfrac{\partial F(\mathfrak{\bar{y}})}{\partial \mathfrak{y}}})\mathfrak{\bar{z}}^\perp$ and $\mathfrak{y}'^\perp=(\overline{\tfrac{\partial F(\mathfrak{{y}})}{\partial \mathfrak{y}}})\mathfrak{{y}}^\perp.$ This implies that
\begin{align}\label{eqn:final-integral-expansion}
&= \int_{D_\delta} \left(e^{\frac{f(x+\delta \chi^\perp,y)}{\varepsilon}} -e^{\frac{f(x,y)}{\varepsilon}}\right)\ \mathrm{d}\nu(y) \\
&= \int_{D_\delta} \left(\frac{f(x+\delta \chi^\perp,y)-f(x,y)}{\varepsilon} +O(\varepsilon^{-1}\delta)\right)e^{\frac{f(x,y)}{\varepsilon}}\ \mathrm{d}\nu(y)\\
&= \int_{G(D_\delta)} \left(\frac{\delta \inner{\nabla  g_{\varepsilon}(F(\mathfrak{y}))}{\tilde{\chi}^\perp}}{\varepsilon}+ O(\varepsilon^{-1}\delta)\right)d\zeta^{g_\varepsilon}(\mathfrak{y}),
\end{align}
where $d\zeta^{g_\varepsilon}(\mathfrak{y})\sim e^{\frac{f(x,y)}{\varepsilon}}\ q(F(\mathfrak{y})) J(F(\mathfrak{y}))(1+O(\delta)) \ \mathrm{d}\mathfrak{y}_1^\perp \cdots \mathrm{d}\mathfrak{y}_{d-1}^\perp\mathrm{d}\mathfrak{y}_d$. Combining \eqref{eqn:df_eps} and \eqref{eqn:final-integral-expansion} we get
\begin{align}
 \frac{\delta}{\varepsilon}\inner{\nabla f_\varepsilon(x)}{\chi^\perp}  &=-\log \frac{\int_{D_\delta} e^{\frac{f(x+\delta \chi^\perp,y)}{\varepsilon}} \ \mathrm{d}\nu(y) }{\int_{D_\delta} e^{\frac{f(x,y)}{ \varepsilon}} \ \mathrm{d}\nu(y) }  +  O(\varepsilon^{-1}\delta)\\
 &= -\log\left( 1+  \frac{\int_{D_\delta} \left(e^{\frac{f(x+\delta \chi^\perp,y)}{\varepsilon}} -e^{\frac{f(x,y)}{ \varepsilon}} \right) \ \mathrm{d}\nu(y)}{\int_{D_\delta} e^{\frac{f(x,y)}{ \varepsilon}} \ \mathrm{d}\nu(y) }\right) +  O(\varepsilon^{-1}\delta)\\
 &= -\frac{\delta}{\varepsilon}\frac{\int_{G(D_\delta)} \inner{\nabla  g_{\varepsilon}(F(\mathfrak{y}))}{\tilde{\chi}^\perp}d\zeta^{g_\varepsilon}(\mathfrak{y})}{\int_{D_\delta} e^{\frac{f(x,y)}{ \varepsilon}} \ \mathrm{d}\nu(y)}+O(\varepsilon^{-1}\delta^2) + O(\varepsilon^{-1}\delta)
\end{align}
and therefore the above equation yields
\begin{align}\label{eqn:derivative identity}
 \inner{\nabla f_\varepsilon(x)}{\chi^\perp} = -\mathbb{E}_{\zeta^{g_{\varepsilon}}}\left[\inner{\nabla  g_{\varepsilon}(y_x)}{\Tilde{\chi}^\perp}\right]+O(\delta)
\end{align}
where we abuse notation and denote the normalized measure by $\zeta^{g_\varepsilon}$. Repeating the argument with $g_\varepsilon$ we get
\begin{align}
\inner{\nabla g_\varepsilon(y)}{\chi^\perp} = -\mathbb{E}_{\zeta^{f_{\varepsilon}}}\left[\inner{\nabla f_{\varepsilon}(x_{y})}{\bar{\chi}^\perp}\right]+O(\delta)
\end{align}
where the above expectation is with respect to the same measure where instead of $g_\varepsilon$ in \eqref{defn:f} we use $f_\varepsilon.$ Then 
\begin{align}
  \inner{\nabla f_\varepsilon(x)}{\chi^\perp} = \mathbb{E}_{\zeta^{f_\varepsilon}}\mathbb{E}_{\zeta^{g_{\varepsilon}}}\left[\inner{\nabla  f_{\varepsilon}}{\chi'}\right]+O(\delta) 
\end{align}
where $\chi'$ is some vector orthogonal to the transport ray $\mathcal{R}_x.$ Then we first claim that $\inner{\nabla f_\varepsilon}{\chi^\perp} $ converges to a constant when $\delta\to 0.$ This follows from the fact that the minimum value of the function $F_\varepsilon(x,\chi^\perp)=\inner{\nabla f_\varepsilon(x)}{\chi^\perp}$ on a compact set $(x,\chi^\perp)\in \mathcal{X}\times \S^{d-1}$, denoted by $F^{\operatorname{min}}_\varepsilon$ satisfies
\begin{align}
    \left|F^{\operatorname{min}}_{\varepsilon}-\mathbb{E}_{\zeta^{f_\varepsilon}}\mathbb{E}_{\zeta^{g_{\varepsilon}}}\left[F_\varepsilon\right] \right| \lesssim \delta \to 0
\end{align}
when $\delta=\varepsilon^{1/2-\kappa}$, $\kappa\in (0,1/2)$ and $\varepsilon\to 0.$ The same argument implies that $\inner{\nabla g_\varepsilon}{\chi^\perp} $ converges as $\varepsilon\to 0.$ The fact that the derivatives of $f_\varepsilon$ and $g_\varepsilon$ converge to the same constant up to a sign follows from sending $\delta=\varepsilon^{1/2-\kappa}\to 0$ in \eqref{eqn:derivative identity}.

\subsection{Proof of Lemma~\ref{lem:constant-on-strips}}
Recall the expression
 \begin{align}\label{eq:Initial_Eq}
 \frac{f_{\varepsilon}(x)-u(x)}{\varepsilon}+\log \int_{\mathcal{Y}} e^{\frac{-I(x,y)+g_{\varepsilon}(y)+u(y)}{\varepsilon}} \mathrm{d}\nu(y) = 0
\end{align}
where $I(x,y) = \|x-y\|-u(x)+u(y).$ Let $\mathcal{R}^{\circ}_x \in T^0_{pij}$ be a part of some ray cluster where we have a local coordinate representation due to Lemma~\ref{lem:lip-change-of-variable}. We first make the following split
\begin{align}
\int_{\mathcal{Y}} e^{-I(x,y)/\varepsilon+\Tilde{g}_\varepsilon} \ \mathrm{d}\nu(y)&= \int_{D_\delta } e^{-I(x,y)/\varepsilon+\Tilde{g}_\varepsilon} \ \mathrm{d}\nu(y) + \int_{D_\delta^c } e^{-I(x,y)/\varepsilon+\Tilde{g}_\varepsilon} \ \mathrm{d}\nu(y)
\end{align}
where $D_\delta = \bigcup_{z\in B_\delta(x)\cap S^i_p} \mathcal{R}_z \cap \mathcal{Y}$ for some $\delta>0$ small enough where $S^i_p$ is same as in Section~\ref{sec:l1-ot}. Like as in the proof of Lemma \ref{lemma:Main}, we introduce
\begin{align}
\Tilde{x}:=G(x)=\mathfrak{x}+\sqrt{\varepsilon}\mathfrak{x}^\perp, \quad \Tilde{y}:=G(y)=\mathfrak{y}+\sqrt{\varepsilon}\mathfrak{y}^\perp, 
\end{align}
for some $\mathfrak{x}^\perp, \mathfrak{y}^\perp \in \mathbb{R}^{d-1} $ which are orthogonal to the direction $\chi$ of $\mathcal{R}_x$. As in the proof of of Lemma~\ref{lemma:Main}, we now have 
\begin{align}
\mathrm{d}\nu(y) & = \varepsilon^{(d-1)/2} q( V(\mathfrak{y}_1.\ldots , \mathfrak{y}_{d-1}) + \mathfrak{y}_d \nu(V(\mathfrak{y}_1.\ldots , \mathfrak{y}_{d-1})))\\ &\times  J(F( V(\mathfrak{y}_1.\ldots , \mathfrak{y}_{d-1}) + \mathfrak{y}_d \nu(V(\mathfrak{y}_1.\ldots , \mathfrak{y}_{d-1})))) (1+O(\delta))d\mathfrak{y}^{\perp}_1 \ldots d\mathfrak{y}^{\perp}_{d-1} d\mathfrak{y}_d\\
I(x,y) &= \|F(\tilde{x}) - F(\tilde{y})\| - \|F(\mathfrak{x}) - F(\mathfrak{y})\| = 
\varepsilon \frac{\|(\overline{\tfrac{\partial F(\mathfrak{x})}{\partial \mathfrak{x}}})\mathfrak{x}^{\perp} - (\overline{\tfrac{\partial F(\mathfrak{y})}{\partial \mathfrak{y}}})\mathfrak{y}^{\perp}\|^2}{2\|F(\mathfrak{x})-F(\mathfrak{y})\|} + O(\varepsilon\delta)\\
\tilde{g}_{\varepsilon}(y) &= \tilde{g}_{\varepsilon}(F(\mathfrak{y})) + \frac{1}{\sqrt{\varepsilon}}\langle \nabla \tilde{g}_{\varepsilon}(F(\mathfrak{y})), (\overline{\tfrac{\partial F(\mathfrak{y})}{\partial \mathfrak{y}}})\mathfrak{y}^{\perp}\rangle + O(\sqrt{\varepsilon})
\end{align}
where $\overline{\tfrac{\partial F(\cdot)}{\partial \cdot}}$ denotes $(d-1)\times (d-1)$ principle submatrix of $\tfrac{\partial F(\cdot)}{\partial \cdot}$. We use the following notation $\bar{\mathfrak{x}}^\perp := (\overline{\tfrac{\partial F(\mathfrak{x})}{\partial \mathfrak{x}}})\mathfrak{x}^{\perp} $ and $\bar{\mathfrak{y}}^\perp := (\overline{\tfrac{\partial F(\mathfrak{y})}{\partial \mathfrak{y}}})\mathfrak{x}^{\perp} $. This yields the following expression after plugging into \eqref{eq:Initial_Eq},
\begin{align}
&\Tilde{f}_\varepsilon(x) =\frac{f_{\varepsilon}(x)-u(x)}{\varepsilon} \\
&= -\log \int_{D_\delta} e^{-I(x,y)/\varepsilon+\Tilde{g}_\varepsilon(y)} \ \mathrm{d}\nu(y) + O(e^{-O(1/\varepsilon)})\\
&= -\log \varepsilon^{(d-1)/2} e^{-O(\sqrt{\varepsilon})+O(\delta)}\int_{\mathcal{R}_x}  \underbrace{\int^{\delta\varepsilon^{-1/2}}_{-\delta\varepsilon^{-1/2}}\cdots \int^{\delta\varepsilon^{-1/2}}_{-\delta\varepsilon^{-1/2}}}_{(d-1)\text{-times}}  e^{- \frac{\|\bar{\mathfrak{x}}^{\perp} - \bar{\mathfrak{y}}^{\perp}\|^2}{2\|x-y\|} - \frac{1}{\sqrt{\varepsilon}} \langle  \nabla \tilde{g}_{\varepsilon}(F(\mathfrak{y})), (\bar{\mathfrak{x}}^\perp - \bar{\mathfrak{y}}^\perp)\rangle}\\
&\quad \times q(y) e^{\tilde{g}_{\varepsilon}(F(\mathfrak{y}))+ \frac{1}{\sqrt{\varepsilon}}\inner{\nabla \tilde{g}_{\varepsilon}(F(\mathfrak{y}))}{\bar{\mathfrak{x}}^\perp}} \ \mathrm{d}\bar{\mathfrak{y}}^\perp_1 \ldots \mathrm{d}\bar{\mathfrak{y}}^\perp_{d-1} d\mathfrak{y}_d + o_{\varepsilon}(1).
\end{align}
As $\varepsilon \to 0$, the $(d-1)$ inner integrals converge to Gaussian integrals that can be simply integrated to yield the following asymptotic relation
\begin{align}\label{eqn:tilde-f-expansion}
\Tilde{f}_\varepsilon(x) = -\log \varepsilon^{(d-1)/2} \int_{\mathcal{R}_x^\circ\cap \mathcal{Y}} & (2\pi \|\bar{\mathfrak{x}}^\perp_d-\bar{\mathfrak{y}}^\perp_d\|)^{\frac{d-1}{2}} e^{\frac{\|\nabla \tilde{g}_{\varepsilon}(F(\mathfrak{y}))\|^2}{2\varepsilon}\|\bar{\mathfrak{x}}^\perp_d-\bar{\mathfrak{y}}^\perp_d\|} \\ &\times e^{\tilde{g}_{\varepsilon}(y)+\frac{1}{\sqrt{\varepsilon}}\inner{\nabla \tilde{g}_{\varepsilon}(F(\mathfrak{y}))}{\bar{\mathfrak{x}}^\perp}}  q(y) \ \mathrm{d}\mathfrak{y}_d + o_\varepsilon(1)
\end{align}
and thus, we define $\Theta^1_\varepsilon$ as a function of the distance along the 
\begin{align}\label{eq:Theta_1}
    \Theta^1_\varepsilon := (2\pi \|\bar{\mathfrak{x}}^\perp_d-\bar{\mathfrak{y}}^\perp_d\|)^{\frac{d-1}{2}}  e^{\frac{\|\nabla \tilde{g}_{\varepsilon}(F(\mathfrak{y}))\|^2}{2\varepsilon}\|\bar{\mathfrak{x}}^\perp_d-\bar{\mathfrak{y}}^\perp_d\|}e^{\frac{1}{\sqrt{\varepsilon}}\inner{\nabla \tilde{g}_{\varepsilon}(F(\mathfrak{y}))}{\bar{\mathfrak{x}}^\perp}}.
\end{align}
Repeating the same argument yields a similar relation for \eqref{eqn:tilde-g}.

\subsection{Proof of Lemma~\ref{lem:entropic-cost-expansion}}
We begin by using the dual formulation \eqref{defn:dual-v_eps} with the potentials $(f_K, g_K)$ attaining the supremum in \eqref{defn:dual-v0} to get the following lower bound
\begin{align}
v_{\varepsilon} & \geq \int_{\mathcal{X}} f_K \ \mathrm{d} \mu+\int_{\mathcal{Y}} g_K \ \mathrm{d} \nu-\varepsilon \log \left(\int_{\mathcal{X} \times \mathcal{Y}} e^{-\frac{I}{\varepsilon}} \mathrm{~d} \mu \otimes \mathrm{d}\nu\right) \\
& =v_0-\varepsilon \log \left(\int_{\mathcal{X} \times \mathcal{Y}} e^{-\frac{I}{\varepsilon}} \ \mathrm{d} \mu \otimes \mathrm{d}\nu\right),    
\end{align}
where $I(x,y)=\|x-y\|-f_K(x)-g_K(y)$ for $(x,y)\in \mathcal{X} \times \mathcal{Y}$. We claim that for all $\varepsilon>0$, there exists a constant $C>0$ (independent of $\varepsilon$) such that the following holds
\begin{align}
\int_{\mathcal{X} \times \mathcal{Y}} e^{-I / \varepsilon} \ \mathrm{d} \mu \otimes \mathrm{d}\nu \leq C \varepsilon^{(d-1)/2}.    
\end{align}
To prove this we consider the covering described in Lemma~\ref{lemma:Main}. By Lemma \ref{lem:support-lemma} it suffices to prove this estimate on $\bigcup_{p\in \mathbb{Q},i,j\in \mathbb{N}} (T^{0}_{pij}\cap \mathcal{X}) \times (T^{0}_{pij}\cap \mathcal{Y}).$ Define for any $p\in \mathbb{Q}$, 
\begin{align}
\Tilde{T}_p := \bigcup_{i,j\in \mathbb{N}} (T^{0}_{pij}\cap \mathcal{X}) \times (T^{0}_{pij}\cap \mathcal{Y}).
\end{align}
Then to avoid double counting we refine this family of sets $\{\Tilde{T}_p\}_{p\in \mathbb{Q}}$ by defining a disjoint family of sets
\begin{align}
\bar{T}_{p_1} &:= \Tilde{T}_{p_1}\\
\bar{T}_{p_k} &:= \Tilde{T}_{p_k}\setminus \bigcup_{q=1}^{k-1}\bar{T}_{p_q},\quad k\geq 2
\end{align}
where $\{{p_n}\}_{n=1}^{\infty}$ is an enumeration of the rational numbers. This gives us a disjoint family of sets and therefore it suffices to estimate
\begin{align}
    \int_{\mathcal{X}\times \mathcal{Y}} e^{-I/\varepsilon} \ \mathrm{d}\mu \otimes \mathrm{d}\nu \leq \limsup_{m\to \infty}\int_{\cup_{n=1}^m \bar{T}_{p_n}} e^{-I/\varepsilon} \ \mathrm{d}\mu \otimes \mathrm{d}\nu.
\end{align}
Now observe that using the covering described in Lemma~\ref{lemma:Main} of the $T^0_{pij}$ by
\begin{align}
  A^k_\delta := \bigcup_{z\in B_\delta(x_k)}  \mathcal{R}_z \cap \mathcal{X},\quad D^k_\delta := \bigcup_{z\in B_\delta(x_k)}  \mathcal{R}_z \cap \mathcal{Y}
\end{align}
where $\{B_\delta(x_k)\}_{k\in \mathbb{N}}$ covers $S^i_p$, we have
\begin{align}
\int_{\cup_{n=1}^m \bar{T}_{p_n}} e^{-\frac{I}{\varepsilon}} \ \mathrm{d}\mu \otimes \mathrm{d}\nu &\leq \int_{\Tilde{T}_{p_m}} e^{-\frac{I}{\varepsilon}} \ \mathrm{d}\mu \otimes \mathrm{d}\nu\\
&\leq \sum_{k} \int_{A^k_{\delta}\times D^k_{\delta^{1-\kappa}}} e^{-\frac{I}{\varepsilon}} \ \mathrm{d}\mu \otimes \mathrm{d}\nu \\
&\quad + \sum_{k}\int_{A^k_{\delta}\times \left(D^k_{\delta^{1-\kappa}}\right)^c} e^{-\frac{I}{\varepsilon}} \ \mathrm{d}\mu \otimes \mathrm{d}\nu.
\end{align}
Then arguing as in Lemma~\ref{lemma:Main}, we have 
\begin{align}
\mathrm{d}\nu(y) & = \varepsilon^{(d-1)/2} q( V(\mathfrak{y}_1.\ldots , \mathfrak{y}_{d-1}) + \mathfrak{y}_d \nu(V(\mathfrak{y}_1.\ldots , \mathfrak{y}_{d-1})))\\ &\times  J(F( V(\mathfrak{y}_1.\ldots , \mathfrak{y}_{d-1}) + \mathfrak{y}_d \nu(V(\mathfrak{y}_1.\ldots , \mathfrak{y}_{d-1})))) (1+O(\delta))d\mathfrak{y}^{\perp}_1 \ldots d\mathfrak{y}^{\perp}_{d-1} d\mathfrak{y}_d\\
I(x,y) &= \|F(\tilde{x}) - F(\tilde{y})\| - \|F(\mathfrak{x}) - F(\mathfrak{y})\| = 
\varepsilon \frac{\|(\overline{\tfrac{\partial F(\mathfrak{x})}{\partial \mathfrak{x}}})\mathfrak{x}^{\perp} - (\overline{\tfrac{\partial F(\mathfrak{y})}{\partial \mathfrak{y}}})\mathfrak{y}^{\perp}\|^2}{2\|F(\mathfrak{x})-F(\mathfrak{y})\|} + O(\varepsilon\delta)
\end{align}
where $\overline{\tfrac{\partial F(\cdot)}{\partial \cdot}}$ denotes $(d-1)\times (d-1)$ principle submatrix of $\tfrac{\partial F(\cdot)}{\partial \cdot}$. We use the following notation $\bar{\mathfrak{x}}^\perp := (\overline{\tfrac{\partial F(\mathfrak{x})}{\partial \mathfrak{x}}})\mathfrak{x}^{\perp} $ and $\bar{\mathfrak{y}}^\perp := (\overline{\tfrac{\partial F(\mathfrak{y})}{\partial \mathfrak{y}}})\mathfrak{x}^{\perp}$. Thus
\begin{align}\label{eqn:first-term}
&\sum_{k}\int_{A^k_{\delta}\times D^k_{\delta^{1-\kappa}}} e^{-\frac{I}{\varepsilon}} \ \mathrm{d}\mu \otimes \mathrm{d}\nu \\
&\lesssim  \varepsilon^{(d-1)/2} \int_{A^k_{\delta}\times (D^k_{\delta^{1-\kappa}}\cap \mathcal{R})} \mathrm{d}\mu(x) \int_{\R^{d-1}}    e^{- \frac{\|\bar{\mathfrak{x}}^{\perp} - \bar{\mathfrak{y}}^{\perp}\|^2}{2\|x-y\|}}  q(y)   d\bar{\mathfrak{y}}^\perp_1 \ldots d\bar{\mathfrak{y}}^\perp_{d-1} d{\mathfrak{y}}_d\\
&\lesssim \varepsilon^{(d-1)/2} \sum_{k}\operatorname{Vol}(A_{\delta}^k)\lesssim \varepsilon^{(d-1)/2} \operatorname{Vol}(\mathcal{X}).
\end{align}
On the other hand, using Lemma~\ref{claim:lower bound} with the assumption that $\operatorname{dist}(d,\mathcal{R}_a)\geq \delta^{1+\kappa}$ for any $(a,d)\in A^k_\delta \times (D^k_{\delta^{1-\kappa}})^c$ with $\delta=\varepsilon^{1/2-\kappa}$ for $\kappa\in (0,1/2)$ small enough to get
\begin{align}\label{eqn:second-term}
  \sum_{k} \int_{A^k_{\delta}\times \left(D^k_{\delta^{1-\kappa}}\right)^c} e^{-\frac{I}{\varepsilon}} \ \mathrm{d}\mu \otimes \mathrm{d}\nu &\lesssim   \sum_{k}\int_{A^k_{\delta}\times \left(D^k_{\delta^{1-\kappa}}\right)^c} e^{-\frac{c_0 \delta^{2+2\kappa}}{\varepsilon}} \ \mathrm{d}\mu \otimes \mathrm{d}\nu \quad \\
  &\lesssim e^{-\frac{c_0 \delta^{2+2\kappa}}{\varepsilon}} \sum_{k}\operatorname{Vol}\left(A^k_\delta \times (D^k_{\delta^{1-\kappa}})^c\right)\\
  &\lesssim e^{-\varepsilon^{-\frac{c_0(1+2\kappa)^2}{2}}}\operatorname{Vol}\left(\mathcal{X}\times \mathcal{Y}\right)
\end{align}
which shows that this term decays exponentially faster than $\varepsilon^{(d-1)/2}$ when $\varepsilon \to 0$. Thus combining \eqref{eqn:first-term} and \eqref{eqn:second-term} we get
\begin{align}
\int_{\mathcal{X}\times \mathcal{Y}} e^{-I/\varepsilon} \ \mathrm{d}\mu \otimes \mathrm{d}\nu &\leq \limsup_{m\to \infty}\int_{\cup_{n=1}^m \bar{T}_{p_n}} e^{-I/\varepsilon} \ \mathrm{d}\mu \otimes \mathrm{d}\nu\\
&\lesssim \limsup_{m\to \infty}\int_{ \Tilde{T}_{p_m}} e^{-I/\varepsilon} \ \mathrm{d}\mu \otimes \mathrm{d}\nu\\
&\lesssim \varepsilon^{(d-1)/2}
\end{align}
as desired.

\begin{remark}
The above computation is simpler when the source and target measures are unit rectangles in $\mathbb{R}^2$, i.e. $\mu=\mathbf{1}_{[0,1]^2}\mathrm{d}x$ and $\nu=\mathbf{1}_{[2,3]\times [0,1]}\mathrm{d}y$. Then
\begin{align}
\int_{B_{\delta}(x)\times B_{\delta}(y)} e^{-E / \varepsilon} \mathrm{d} \mu \otimes \nu &\leq \int_{B_{\delta}(x)\times B_{\delta}(y)}  e^{-\frac{|x_2-y_2)^2}{2\varepsilon |x_1-y_1|}+ \frac{O(\delta^2)}{\varepsilon}} \ \mathrm{d}x_1 \mathrm{d}y_1 \mathrm{d}x_2 \mathrm{d}y_2 \\
&\leq \int \sqrt{2\pi \varepsilon |x_1-y_1|} e^{-O(\delta^2)/\varepsilon} \mathrm{d}x_2 \mathrm{d}x_1 \mathrm{d}y_1\\
&\leq C \varepsilon^{1/2}  \delta^4 \leq C \sqrt{\varepsilon} \operatorname{Vol}(B_{\delta}(x)\times B_{\delta}(y)).
\end{align}
Summing this over all such balls we get
\begin{align}
\int_{\cup_{i=1}^n B^i_{\delta}(x)\times B^i_{\delta}(y)} e^{-E / \varepsilon} \mathrm{d} \mu \otimes \nu \leq C \sqrt{\varepsilon}
\end{align}
since the volume of $\delta$-balls are controlled by the total volume of the support of the measures $\mu$ and $\nu.$ We estimate the remainder $\mathcal{X}\times \mathcal{Y} \setminus \cup_{i=1}^n B^i_{\delta}(x)\times B^i_{\delta}(y)$ by a crude estimate of the form $e^{-\frac{I(x,y)}{\varepsilon}}\leq e^{-\frac{C}{\varepsilon}}$ which then gives the estimate
\begin{align}
    \int_{\mathcal{X}\times \mathcal{Y}} e^{-E/\varepsilon} \mathrm{d}\mu \otimes \mathrm{d}\nu \leq C \sqrt{\varepsilon}
\end{align}
as desired.
\end{remark}

\section{Auxiliary Results}
\begin{lemma}\label{lem:rate-bound}
Let $A^k_\delta$ and $(D^k_{\delta^{1-\kappa}})^c$ be subsets as defined in Lemma~\ref{lemma:Main}. Then there exists a constant $C>0$ and $\varepsilon_0>0$ such that
\begin{align}
\pi^{\varepsilon}(A^k_\delta \times (D^k_{\delta^{1-\kappa}})^c) \lesssim \inf_{(a,d)\in A^k_\delta \times (D^k_{\delta^{1-\kappa}})^c} e^{-\frac{I(a,d)}{\varepsilon}}
\end{align}
for $\varepsilon\leq \varepsilon_0$. 
\end{lemma}
\begin{proof}
For convenience, we drop the superscript $k$ in the definition of sets $A^k_\delta$ and $(D^k_{\delta^{1-\kappa}})^c$. Fix $\eta>0$ and $(a, d) \in A_{\delta}\times D_{\delta^{1-\kappa}}^c$. Recall the definition of the rate function $I$ as in Lemma 4.2~\cite{entropic-cyc-ld}; given a c-cyclically monotone set $\emptyset \neq \Gamma \subseteq \mathcal{X} \times \mathcal{Y}$, define
$$
I(x, y):=\sup _{\ell \geq 2} \sup _{\left(x_i, y_i\right)_{i=2}^\ell \subset \Gamma} \sup _{\sigma \in \Sigma(\ell)} \sum_{i=1}^\ell \|x_i-y_i\|-\sum_{i=1}^\ell \|x_i - y_{\sigma(i)}\|
$$
where $\left(x_1, y_1\right):=(x, y)$ and $\Sigma(\ell)$ is the set of permutations of $\{1,\cdots,\ell\}$. Thus there exist $\ell \geq 2$ and $\left(a_i, d_i\right)_{i=2}^\ell \subset \mathcal{X}\times \mathcal{Y}$ such that
\begin{align}
\sum_{i=1}^\ell \| a_i -d_i \|-\sum_{i=1}^\ell \| a_i- d_{i+1}\| >I_\eta(a, d)-\eta / 2    
\end{align}
where $\left(a_1, d_1\right):=(a, d)$ and $I_\eta(a, d):=I(a, d) \wedge \eta^{-1}$, and we replace the rate function $I$ by $I_\eta$ in the case when $I=+\infty.$ Then applying Lemma 4.1 in \cite{entropic-cyc-ld} yields a ball $B_r(a, d)$ with
\begin{align}
    \pi^{\varepsilon}(B_r(a,d)) \leq C e^{-\frac{I_{\eta}(a,d)-\eta}{\varepsilon}}
\end{align}
for $\varepsilon \leq \varepsilon_0$ with $\varepsilon_0$ and $r>0$ small enough. Covering the compact set $A_{\delta}\times D_{\delta^{1-\kappa}}^c$ by a finite number of balls $B_r(a,d)$ with some small enough fixed radius $r>0$ to get
\begin{align}
    \pi^{\varepsilon}(A_{\delta}\times D_{\delta^{1-\kappa}}^c) \lesssim  \inf_{(a,d)\in A_{\delta}\times D_{\delta^{1-\kappa}}^c} e^{-\frac{I_{\eta}(a,d)-\eta}{\varepsilon}}.
\end{align}
Since the above estimate holds for arbitrary $\eta>0$, we get the desired inequality.
\end{proof}

\begin{figure}
\begin{tikzpicture}[scale=2.5, line join=round, line cap=round]

\definecolor{mywhite}{rgb}{0,0,0}
\definecolor{myblue}{rgb}{0,0.6,1}
\definecolor{mygreen}{rgb}{0,0.8,0.4}
\definecolor{myorange}{rgb}{1,0.5,0}
\definecolor{myred}{rgb}{1,0.2,0.2}
\definecolor{myviolet}{rgb}{0.6,0.2,1}

\fill[white] (4,0,0) -- (4,1.8,0) -- (4,1.8,1.8) -- (4,0,1.8) -- cycle;

\draw[myorange, thick, dashed] (0.5,0,0) -- (0.5,1,0) -- (0.5,1,1) -- (0.5,0,1) -- cycle;

\draw (0.5,0.5,0.5) node {.} coordinate[label=below: $a$] (a);

\draw[myblue, dashed] (0.5,0.5,0.5) -- (3.5,0.5,0.5) ;
\node at (2.5,0.45,0.5) {$\mathcal{R}_a$};

\draw[myred, thick] (0.5,0.5,0.5) ellipse [x radius=0.1, y radius=0.3];
\node at (-0.25,-0.1,0.5) {$B_{\delta^{1+\tau \kappa}}(a)$};
\draw[->] (0,0,0.5) -- (0.4,0.34,0.5);

\draw (0.5,0.8,0.5) node {.} coordinate[label=above: $a'$] (ap);;

\draw[myblue, dashed] (0.5,0.8,0.5) -- (3.5,0.8,0.5);
\node at (2.5,0.87,0.5) {$\mathcal{R}_{a'}$};

\draw[mygreen, dashed] (0.5,0.5,0.5) -- (3.5,1.3,0.3);

\draw[myorange, thick, dashed] (3.5,0,0) -- (3.5,1.5,0) -- (3.5,1.5,1.5) -- (3.5,0,1.5) -- cycle;

\draw (3.5,0.5,0.5) node {.} coordinate[label=right: $d'$] (dp);

\draw (3.5,0.8,0.5) node {.} coordinate[label=right: $d''$] (dpp);

\draw (3.5,1.3,0.3) node {.} coordinate[label=right: $d$] (d);

\draw[black] (0.8,0.5,0.5) arc (10:50:0.13);
\node at (0.85,0.55,0.5) {$\alpha$};
 
\draw[black, thick] (current bounding box.south west) rectangle (current bounding box.north east);
\end{tikzpicture}
\caption{Pictorial depiction of the points $a,a',d, d',d'',d$ constructed in the proof of Lemma~\ref{claim:lower bound} when $\mathcal{X}$ and $\mathcal{Y}$ are two unit cubes where one is a linear translation of other and $\mu,\nu$ are uniform measures on $\mathcal{X},\mathcal{Y}$ respectively.}
\end{figure}
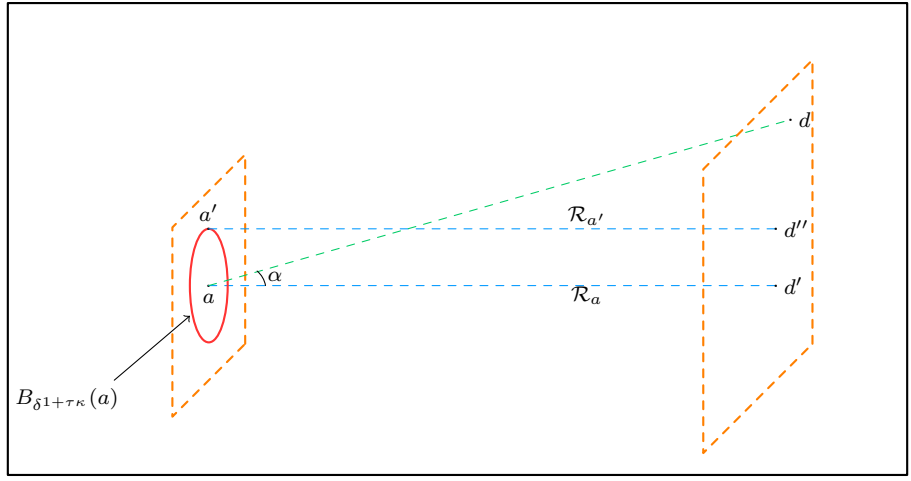

\begin{lemma}\label{claim:lower bound}
Let $(a,d)\in A_\delta\times D_\delta\cap \{(x,y) \in \mathcal{X}\times \mathcal{Y}:I(x,y)\neq 0\}$ be such that $\mathrm{dist}(d,\mathcal{R}_a)\geq \delta^{1+\kappa}$ for some small $\kappa>0$ where $\mathcal{R}_a$ is the transport ray from $a$. Then we have  
\begin{align}\label{eq:Idelta}
    I(a,d)\geq c\delta^{2+2\kappa}.  
\end{align}
\end{lemma}
\begin{proof}
Fix $\big(a, d'\big) \in \mathcal{R}_a$ where $\mathcal{R}_a$ is the transport ray connecting $a$ and $\mathcal{Y}$ and $d' = \operatorname{argmin}_{y \in \mathcal{R}_a} \|y- d\|$. Let $B_{\delta^{1+\tau\kappa}}(a)$ be a ball of radius $1+\tau\kappa$ around $a$ for some $\tau>1$ such that $0<\tau-1\ll 1$. Define 
\[v = \frac{a-d}{\|a-d\|} - \frac{a-d'}{\|a-d'\|}\]
and choose $a' \in \partial B_{\delta^{1+\tau\kappa}}(a)$ such that
$ a- a'$ is parallel to $v$. Let $\mathcal{R}_{a'}$ be the transport ray from $a'$ and let $d'' \in \mathcal{R}_{a'}\cap \mathcal{Y}$ such that $d''$ and $d'$ resides on the same level set of the Kantorovich potential $u$. 
Due to Lemma 4.2~\cite{entropic-cyc-ld}, it suffices to show the following in order to prove \eqref{eq:Idelta}:
\begin{align}\label{eq:cyclic_monotonic_est}
\|a- d\|-\|a'- d\|+\|a' -d''\|-\|a - d''\|\geq c\delta^{2+(1+\tau)\kappa}.
\end{align}
Notice that 
\begin{align}
 \| & a- d\|-\|a'- d\|+\|a' -d'\|-\|a - d'\| \\
& = \Big\langle \frac{a-d}{\|a-d\|} - \frac{a-d'}{\|a-d'\|} + O(\|d'-d''\|), a-a'\Big\rangle  + O(\|a-a'\|^2)\label{eq:InnerProd}.
\end{align}
Define $\nu_{\vec{d'a}} =  d'-a/\|a-d'\|$ and denote $\nu^{\perp}_{ \vec{d'a}}$ to be the direction perpendicular to $\nu_{\vec{d'a}}$ which is in the same plane as $a-d$. Let $\alpha\in (0,\pi/2)$ be such that $a-d/\|a-d\| = (-\cos \alpha) \nu_{\vec{d'a}}+ (\sin \alpha) \nu^{\perp}_{\vec{d'a}}$. 
Notice that $a-d'/\|a-d'\| = -\nu_{\vec{d'a}}$. As a result, we can write 
\begin{align}
 \frac{a-d}{\|a-d\|} - \frac{a-d'}{\|a-d'\|} &= (-\cos \alpha + 1)\nu_{\vec{aa'}}+ (\sin \alpha) \nu^{\perp}_{\vec{aa'}}, \\ \Big\|\frac{a-d}{\|a-d\|} - \frac{a-d'}{\|a-d'\|}\Big\| &= 2\sin(\alpha/2).
\end{align}
Notice that $\sin(\alpha)= \|d-d'\|/\|a-d\|\geq c\|d-d'\|\geq c\delta^{1+\kappa}$ some positive constant $c\geq \mathrm{dist}(\mathcal{X}, \mathcal{Y})^{-1}$. As a result, we can bound the right-hand side as
\begin{align*}
    \text{r.h.s. of \eqref{eq:InnerProd}} &\geq 2\sin(\alpha/2)\|a-a'\| + O(\|d-d''\|\|a-a'\| + \|a-a'\|^2).
\end{align*}
Notice that the first term on the right-hand side is bounded below by $c\delta^{2+(1+\tau)\kappa}$ for some $c>0$. On the other hand side, $\|a-a'\|$ is bounded above by $c'\delta^{2(1+\tau \kappa)}$. To complete the proof of \eqref{eq:cyclic_monotonic_est}, it suffices to show 
\begin{align}
    \|d'-d''\| \lesssim \|a-a'\|.
\end{align}
Let $d''$ and $d'$ be any two points lying on the same level set $\{u=q\}$ emanating from the points $\Tilde{a}$ and $a$ respectively. Then
\begin{align}
d'' = \Tilde{a} + \nu(\Tilde{a}) \|d''-\Tilde{a}\|,\quad d' = a+ \nu(a) \|d'-a\|. 
\end{align}
Since $\Tilde{a},a\in \{u=p\}$ and $d',d'' \in \{u=q\}$ we have
\begin{align}
\left\|d^{\prime \prime}-\tilde{a}\right\|=u\left(d^{\prime \prime}\right)-u(\tilde{a}) =u\left(d^{\prime}\right)-u(a)=\left\|d^{\prime}-a\right\|.
\end{align}
Using Lemma \ref{lem:lip-control} we have
\begin{align}
\|d'-d''\| &\leq \|\Tilde{a}-a\|+\|d'-a\|\|\nu(\Tilde{a}-\nu(a)\|\\
&\leq \|\Tilde{a}-a\| + \frac{C}{\sigma} \|d'-a\|\| \|\Tilde{a}-a\|\\
&\leq \|\Tilde{a}-a\|\left( 1+ \frac{C}{\sigma} \|d'-a\|\right)\\
&\lesssim \|\Tilde{a}-a\|,
\end{align}
where constant above depends on $C,\sigma$ and $\operatorname{diam}(\mathcal{X}),\operatorname{diam}(\mathcal{Y})$ and $\operatorname{dist}(\mathcal{X},\mathcal{Y}).$ Furthermore, by using the Lipschitz change of coordinates as in Lemma \ref{lem:lip-change-of-variable}, we argue that
\begin{align}
    \|\Tilde{a}-a\| \lesssim \|a'-a\|.
\end{align}
To see this, we note that
\begin{align}
\|\Tilde{a}-a'\|\leq \operatorname{dist}\left((U(\Tilde{a}),0),(U(a'),u(a')-u(\Tilde{a}))\right)\lesssim \|a-a'\|
\end{align}
where the constant $C$ depends on the Lipschitz constant of the change of coordinate map $U$ as in Lemma \ref{lem:lip-change-of-variable}. This, in turn, implies that
\begin{align}
\|\Tilde{a}-a\| &\leq \|\Tilde{a}-a'\|+\|a'-a\| \lesssim \|a'-a\|.
\end{align}
Therefore, we get $\|d'-d''\|\lesssim \|a-a'\|$ as desired. This implies $\|d'-d''\|\lesssim \|a'-a\|$ which completes the proof.
\end{proof}

\begin{lemma}\label{lem:support-lemma}
Let $\pi \in \Gamma(\mu,\nu)$ be an optimal coupling for \eqref{eqn:kant-l1}. Then
\begin{align}
     \pi\left(\bigcup_{p\in \mathbb{Q},i,j\in \mathbb{N}}(T^\circ_{pij}\cap \mathcal{X}\times T^\circ_{pij}\cap \mathcal{Y})\right) = 1.
\end{align}
\end{lemma}
\begin{proof}
From Lemma 9 in \cite{mcann-caff-feldman}, we know that $\mathcal{X}\cup \mathcal{Y}\subset T_0 \cup T_1$. Since $\mathcal{X}\cap \mathcal{Y}=\emptyset$, the set of points rays of length zero is empty, that is, $T_0 = \emptyset.$ Next since $\mathcal{X}\subset T_1 \cap \mathcal{X}$ and $\mathcal{Y}\subset T_1 \cap \mathcal{Y}$ we get that for any $\pi \in \Gamma(\mu,\nu)$
\begin{align}
1=\mu(\mathcal{X})\nu(\mathcal{Y})=\pi(\mathcal{X}\times \mathcal{Y})\leq \pi(T_1^\circ \cap \mathcal{X}\times T_1^\circ \cap \mathcal{Y}),
\end{align}
where we used the fact that $T_1^\circ = T_1\setminus \mathcal{E}$ and $\mathcal{E}$ denotes the set of ray ends which by Lemma 25 in \cite{mcann-caff-feldman} have zero Lebesgue measure.
Next, since 
\begin{align}
T_1^\circ \cap \mathcal{X} = \bigcup_{p\in \mathbb{Q},i,j\in \mathbb{N}}T^\circ_{pij}\cap \mathcal{X},\quad T_1^\circ \cap \mathcal{Y} = \bigcup_{q\in \mathbb{Q},k,l\in \mathbb{N}}T^\circ_{qkl}\cap \mathcal{Y}
\end{align}
we have
\begin{align}
1 &=\pi\left(\bigcup_{p\in \mathbb{Q},i,j\in \mathbb{N}}T^\circ_{pij}\cap \mathcal{X}\times \bigcup_{q\in \mathbb{Q},k,l\in \mathbb{N}}T^\circ_{qkl}\cap \mathcal{Y}\right) \\
&\leq \pi\left(\bigcup_{p\in \mathbb{Q},i,j\in \mathbb{N}} \bigcup_{q\in \mathbb{Q},k,l\in \mathbb{N}}(T^\circ_{pij}\cap \mathcal{X})\times  (T^\circ_{pij}\cap T^\circ_{qkl}\cap \mathcal{Y})\right) \\
&\quad + \pi\left(\bigcup_{p\in \mathbb{Q},i,j\in \mathbb{N}} \bigcup_{q\in \mathbb{Q},k,l\in \mathbb{N}}(T^\circ_{pij}\cap \mathcal{X})\times  ((T^\circ_{qkl}\setminus T^\circ_{pij})\cap \mathcal{Y})\right)\\
&\leq \pi\left(\bigcup_{p\in \mathbb{Q},i,j\in \mathbb{N}}(T^\circ_{pij}\cap \mathcal{X}\times T^\circ_{pij}\cap \mathcal{Y})\right) + \pi(Z)
\end{align}
where we define the set $Z$ as
\begin{align}
    Z:= \bigcup_{p\in \mathbb{Q},i,j\in \mathbb{N}} \bigcup_{q\in \mathbb{Q},k,l\in \mathbb{N}}(T^\circ_{pij}\cap \mathcal{X})\times  ((T^\circ_{qkl}\setminus T^\circ_{pij})\cap \mathcal{Y}).
\end{align}
We will now show that $\pi(Z)=0.$ Indeed, if $x\in T^\circ_{pij}\cap \mathcal{X}$ then there exists a transport ray such that $x\in \mathcal{R}_1^\circ$ with endpoints $(a_x,b_x)$ intersecting the level set $S^i_p$. On the other hand if $y\in (T^\circ_{qkl}\setminus T^\circ_{pij})\cap \mathcal{Y}$ then $y\in \mathcal{R}_2^\circ$ some transport ray with endpoints $(a_y,b_y)$ that is disjoint for $\mathcal{R}_1$, i.e. $\mathcal{R}_1^\circ \cap \mathcal{R}_2^\circ = \emptyset.$ We claim that $u(x)-u(y) \neq \|x-y\|.$ Suppose to the contrary $u(x)-u(y) = \|x-y\|$, then we will show that $a_x,x,y,b_y$ are collinear. This is because $u(x)-u(y) = \|x-y\|$, $u(x)= u(a_x) - \|x-a_x\|$ and $u(y)=u(b_y)+\|y-b_y\|$ imply 
\begin{align}
\|a_x-x\| + \|x-y\| + \|y-b_y\| = u(a_x) - u(b_y) \leq \|a_x-b_y\|.
\end{align}
Combining this with the triangle inequality $\|a_x-b_y\| \leq \|a_x-x\| + \|x-~y\| + \|y-b_y\|$ yields equality
\begin{align}
    \|a_x-x\| + \|x-y\| + \|y-b_y\| = \|a_x-b_y\|
\end{align}
which implies that $a_x,x, y,b_y$ are collinear which contradicts $\mathcal{R}_1^\circ \cap \mathcal{R}_2^\circ =~ \emptyset.$ Thus $u(x)-u(y)\neq \|x-y\|.$ Define 
\begin{align}
\Gamma:=\left\{(x, y) \in \mathcal{X} \times \mathcal{Y}: u(x)-u(y)=\|x-y\|\right\} .
\end{align}
Since $\pi\in \Gamma(\mu,\nu)$ is an optimal coupling for \eqref{eqn:kant-l1} we have that $\pi(\Gamma)=1.$ Therefore
\begin{align}
\pi(Z)\leq \pi(\mathcal{X}\times \mathcal{Y}\setminus \Gamma) = 0    
\end{align}
implying that $\pi(Z)=0.$ Thus,
\begin{align}
 \pi\left(\bigcup_{p\in \mathbb{Q},i,j\in \mathbb{N}}(T^\circ_{pij}\cap \mathcal{X}\times T^\circ_{pij}\cap \mathcal{Y})\right) = 1.   
\end{align}
\end{proof}
\bibliographystyle{alpha}
\bibliography{refs}
\end{document}